\documentclass[a4paper]{article}

\usepackage{fullpage}
\usepackage{comment}

\usepackage{xcolor}

\usepackage{amsthm,amsmath,amssymb,mathtools}
\usepackage[hidelinks]{hyperref}
\usepackage[nameinlink]{cleveref}

\usepackage{enumitem}

\newtheorem{theorem}{Theorem}[section]
\newtheorem{lemma}[theorem]{Lemma}
\newtheorem{corollary}[theorem]{Corollary}

\theoremstyle{definition}
\newtheorem{remark}[theorem]{Remark}

\usepackage{todonotes}

\newcommand{\Rd}{{\mathbb{R}^d}}
\newcommand{\dRL}{{^{RL}\partial_t^\alpha}}

\newcommand{\asympint}[2]{\overset {(#1,#2)}\asymp }

\usepackage[backend=biber,url=false,giveninits,maxbibnames=10]{biblatex}
\AtEveryBibitem{%
\clearfield{isbn}
\clearfield{issn}
\clearfield{month}
\clearfield{series}
\clearfield{urldate}
\clearfield{pubstate}
\clearfield{url}

\ifentrytype{book}{}{%
\clearlist{publisher}
\clearname{editor}
}
}

\DeclareMathOperator{\fixedpoint}{\mathfrak K}

\usepackage{authblk}

\title{Self-similar solutions to the \\
time-fractional Porous-Medium Equation}
\author[1,2]{David Gómez-Castro}
\author[3]{\L{}ukasz P\l{}ociniczak}
\author[1]{Juan Luis Vázquez}
\affil[1]{Universidad Autónoma de Madrid}
\affil[2]{Instituto de Ciencias Matemáticas, CSIC}
\affil[3]{Wrocław University of Science and Technology}

\date{}

\begin{document}

\maketitle

\begin{abstract}
	We show the existence of self-similar solutions with constant finite mass to the time-fractional Porous-Medium Equation for all spatial dimensions $d \ge 1$ and all exponents $m>m_c=(d-2)_+/d$. This range is optimal. We find two types of solution depending on the exponent: compactly supported solutions in the slow-diffusion range $m > 1$ and positive solutions with heavy tails in the sub-critical fast-diffusion range $m_c < m < 1$. The self-similar solutions in the linear case $m=1$ were already known explicitly obtained by the Fourier transform, and we discuss their properties in our settings and the limit $m \to 1$.

\end{abstract}

\hfill {\sl Dedicated to the memory of G. I. Barenblatt}

\section{Introduction and main results}
In this paper, we discuss the nonnegative self-similar solutions of the nonlinear time-fractional diffusion equation for $m > 0$ and $\alpha \in (0,1)$ given by
\begin{equation}
	\label{eq:main}
	\partial_t^\alpha u = \Delta (u^{m})
\end{equation}
for $t > 0$ and $x \in \mathbb R^d$ for any $d \ge 1$.
Here $\partial_t^\alpha$ is the Caputo derivative
\begin{equation}
	\label{eq:Caputo}
	\partial_t^\alpha u(t,x) = \int_0^t \frac{(t-s)^{-\alpha}}{\Gamma(1-\alpha)} \partial_t u(s,x) ds.
\end{equation}
For $\beta>0$, the left-sided Riemann--Liouville fractional integral is defined by
\[
	(I_\beta f)(t)=\frac{1}{\Gamma(\beta)}\int_0^t(t-s)^{\beta-1}f(s)\,ds.
\]
For $0<\alpha<1$, the Riemann-Liouville derivative is $\dRL f=\frac{d}{dt}I_{1-\alpha}f$, with the derivative taken in the distributional sense when necessary. For absolutely continuous $f$, the Caputo derivative in \eqref{eq:Caputo} is $\partial_t^\alpha f=I_{1-\alpha}f'$.
The solutions start from a Dirac delta $u(0,\cdot) = M \delta_0$ in the sense of initial traces:
\begin{equation}
	\label{eq:initial Dirac}
	\operatorname*{ess\,lim}_{t \to 0^+} \int_{\Rd} u(t,x) \varphi(x) dx = M\varphi(0),
\end{equation}
for all $\varphi \in C^\infty_c (\Rd)$. We look for nonnegative self-similar solutions $u(t,x)$ of the equation that have a finite mass for all $t>0$, and, moreover, they conserve for all times the initial mass $M>0$ in the sense that
\begin{equation}
	\label{eq:mass conservation}
	\int_{\mathbb R^d} u(t,x) dx = M \quad \text{for all } t > 0.
\end{equation}
By self-similar solution, we refer to space-time function, which can be re-scaled to a profile $U$ in the sense that
\begin{equation}\label{eq:ssf}
	u(t,x) = t^{-a} U(t^{-b} |x|),
\end{equation}
for similarity exponents $a,b$ to be appropriately chosen.
Our main contribution is the study of the existence and key properties of $U$ and its dependence on parameters $\alpha,m,d,$ and mass $M$.

\medskip

The classical case $\alpha = 1$ was the source of a significant body of work in the twentieth century. The linear diffusion case $m = 1$ is the most famous, where $u$ is the heat kernel and $u$ is the Gaussian profile. The slow-diffusion $m > 1$ solutions lead to the compactly-supported Barenblatt solutions \cite{MR1426127} (see also \cite{MR46217,zel1950towards}), which started the study of the theory of free boundaries in non-linear diffusion.
These naturally arise in many physical contexts, such as gas flow through heterogeneous media, water percolation in the soil \cite{bear2013dynamics}, or even in biology \cite{vafai2010porous} and medicine \cite{lavrova2021barenblatt}. The nonlinear character of the diffusivity helps to model such physically important features as finite speed of propagation and sharp interfaces. The former is a central characteristic of fluid flow through porous media where disturbances do not propagate with an infinite speed (as modeled by the classical heat equation). Moreover, Barenblatt profiles frequently act as attractors to which solutions starting from general initial conditions converge asymptotically \cite{vazquez2014barenblatt}. In the fast-diffusion case $0 < m < 1$ we have self-similar power-type solutions. In this setting, one must distinguish the mildly-singular fast-diffusion case $m_c < m < 1$, where $m_c = (1 - 2/d)_+$, in which finite-mass self-similar solutions exist. In the very-singular fast-diffusion $0 < m < m_c$ the self-similar solutions have infinite mass. In the very-singular case, there is still a theory of backward-in-time self-similarity, which we will not discuss here.
We also discuss the mesa problem, i.e. the limit as $m\to \infty$, which is well-understood in the classical case $\alpha = 1$ \cite{MR983523,MR883709}.
For details on the Porous-Medium equation, we refer the reader to \cite{Vazquez2006,Vazquez2006a}.
For a general discussion on self-similarity and mass conservation see \cite{vazquez2025surveymassconservationselfsimilarity}.

\medskip

The theory of existence and uniqueness and key properties for the fractional case $\alpha \in (0,1)$ is an active field of research.
The case of linear diffusion, that is, for $m = 1$, the existence was proven in \cite{Zacher2009}, and other properties in \cite{vergara2015,dier2020,cortazarHeatEquationMemory2021}.
For $m > 1$ the closest literature that deals with the existence and properties is
\cite{djidaNonlocalTimeporousMedium2019,DjidaNietoArea2019}
for the nonlocal pressure case where $\Delta u^m$ is replaced by $\operatorname{div}(m u^{m-1} \nabla p)$ and $(-\Delta)^s p = u$. Our problem represents the $s = 0$ case.

\medskip

The theory of self-similar solutions for the fractional-in-time problem with linear diffusion $m = 1$ is known \cite{MR1061448, MR2047909,MR3631303} where the Fourier transform of the self-similar solutions is of closed form given by Mittag-Leffler functions (see \eqref{eq:m=1 self-similar}).
The self-similar solutions for $m > 1$ and $d=1$ were studied in \cite{CaballeroOkrasinska-PlociniczakPlociniczakSadarangani2025} by arriving at an integral equation (see \eqref{eq:profile equation} below) and applying Schauder fixed-point arguments.
A numerical scheme for $m > 1$ and $d=1$ is presented in \cite{Plociniczak2019}.

\medskip

There are several other different directions of problems with fractional-in-time diffusion to cite a few:
with fractional pressure in $\mathbb R^d$ we have \cite{allen2016}.
The study of the case $m > 1$ in bounded domains is done in \cite{bonforteTimeFractionalPorousMedium2024}.
In addition, the abstract approach of flows in the Hilbert space was undertaken in \cite{akagi2019fractional} where both the linear and nonlinear versions of the main equation have been considered (along with other generalizations).
Also, see the latest developments in \cite{akagi2025time}.
Moreover, several existence and uniqueness results for a general nonlinear parabolic problem with the Caputo derivative have been proved in \cite{topp2017existence} via the viscosity solution framework.

\paragraph{Outline of main results.}
We provide a complete study of self-similar profiles for $d \ge 1$, $\alpha \in (0,1)$, $m > m_c = (\tfrac{d-2}d)_+$ that extends the previous results for $m = 1$ in \cite{MR1829592,dier2020}, and $m > 1$ and $d=1$ in \cite{CaballeroOkrasinska-PlociniczakPlociniczakSadarangani2025}.
Similarly to \cite{CaballeroOkrasinska-PlociniczakPlociniczakSadarangani2025} we arrive at an integral equation \eqref{eq:profile equation}, where we have significantly simplified the expression of the kernel in \eqref{eq:Q} even in dimension $d=1$, and we solve it using the method of sub and super-solutions in all cases.
We can summarize the results presented later in this section as follows:
\begin{itemize}
	\item \textbf{Existence} of profiles for all $m>m_c$.

	\item \textbf{Uniqueness} of profiles for $m\ge 1$.

	\item \textbf{Notion of weak solution} satisfied by the self-similar solution.

	\item \textbf{Sharp behaviour at $z=0$.}
	      For all $m > m_c$ we see that $U$ is bounded in $d=1$ and $U^m$ behaves like the Newtonian potential for $d \ge 2$. The critical exponent $m_c$ is optimal.

	\item \textbf{Compact support for $m > 1$} where solutions behave like a simple power near the free bundary. The sharp constant is provided.

	\item \textbf{Tail behavior for $m \in (m_c,1)$}
	      The solutions have power-like tails when $m_c < m < 1$, which can easily be computed by scaling, and they form an asymptotic fan around a very-singular solution of infinite mass.
	      The exponential tail of the linear case $m=1$ is quite complicated. On the contrary, for $m< 1$ the tail decays like a simple power.
	      We match these cases by proving the limit $m \to 1$.

	\item \textbf{Mesa limit $m\rightarrow \infty$} where solutions converge to a characteristic function of the ball of radius $(M/|B_1|)^{1/d}$, as in the classical case $\alpha=1$.

	\item \textbf{Classical limit $\alpha \to 1$} where we recover the usual Barenblatt profiles.

	\item \textbf{The numerical scheme} for the case $m>1$ is constructed by Picard's iterations. We illustrate the behaviors with respect to $\alpha$ and $m$ in \Cref{fig:limit alpha->1,fig:mesa}.
\end{itemize}

\paragraph{Notation.} The self-similar solution has 4 parameters: $\alpha,m,d,M$. For each result and proof, the value of $d$ will not change, so we will not make this dependence explicit. When we intend to change the value of $\alpha,d,M$ through a proof statement, we will make this dependence explicit $U = U_{\alpha,m,M}$. Often we will drop the sub-indices in the interest of clarity. We will denote generic constants we will use the notation $c$ and $C$, and when their dependences can be inferred from the context we will.

\subsection{An integral equation for the self-similar solutions}

Computing the integral, we can show that the conservation of mass \eqref{eq:mass conservation} is equivalent to $a = db$ and
\begin{equation}
	\label{eq:profile mass}
	|\partial B_1| \int_0^\infty U(\rho) \rho^{d-1} d \rho  = M.
\end{equation}
We will show that the scaling factors of the equation are given by
\begin{equation}
	\label{eq:self-similar exponents}
	a = \frac{\alpha d}{2 + d(m-1)} , \quad b = \frac{\alpha}{2+d(m-1)}.
\end{equation}
Notice that $a, b$ are positive if and only
\begin{equation}
	m>m_c \coloneqq (1-\tfrac{2}{d})_+.
\end{equation}
This restriction on $m$ will be kept in this paper.
Formally, if $u$ is classical solution of \eqref{eq:main} of the form \eqref{eq:ssf} then we will show that $U$ satisfies the integral operator equation
\begin{equation}
	\label{eq:profile equation}
	U(z)^{m} = \int_z^{+\infty} K(z,\rho) \, U(\rho) d\rho,
\end{equation}
with a non-negative kernel $ K(z,\rho)$ given for $\rho \ge z$ by
\begin{equation}
	\label{eq:Q}
	K(z,\rho)= \rho \, Q(\tfrac{z}{\rho}), \quad Q(\eta) =
	\frac{1}{\Gamma(1-\alpha)}\int_\eta^1(1-\sigma^{\frac 1 b})^{-\alpha} \sigma^{1-d} d\sigma.
\end{equation}
Notice that $K(z,z) = 0$ and $K(z,\rho) > 0$ if $z < \rho$.
This formulation is a key tool in our paper.

\subsection{Existence of self-similar solutions for $m > m_c$}
Our first result is the existence of self-similar solutions and their characterization as weak solutions.
\begin{theorem}
	\label{thm:existence}
	Let $d \ge 1$, $\alpha \in (0,1)$, $m > m_c$, and $M > 0$. Then, there exists a canonical profile $\mathcal U_{\alpha,m} \in C((0,+\infty))$ such that $z^{d-1} \mathcal U_{\alpha,m}(z) \in L^1(0,\infty)$ and, if we let,
	\begin{equation}
		\label{eq:rescaling profile to correct mass}
		U_{\alpha,m,M}(z) = A \mathcal U_{\alpha,m}(A^{-\frac{m-1}{2}} z), \quad  A(\alpha,m,M)^{1 + \frac{(m-1)d}{2}} = \frac{M}{|\partial B_1| \int_0^\infty \mathcal U_{\alpha,m}(\rho) \rho^{d-1} d \rho}.
	\end{equation}
	the following properties hold:
	\begin{enumerate}
		\item
		      $\mathcal U_{\alpha,m}$ and $U_{\alpha,m,M}$ are solutions to \eqref{eq:profile equation} and $U_{\alpha,m,M}$ satisfies \eqref{eq:profile mass}.
		\item
		      Letting $U = U_{\alpha,m,M}$, the self-similar solution $u = u_{\alpha,m,M}$ given by \eqref{eq:ssf}
		      is a very weak solution of \eqref{eq:main} in the sense that
		      \begin{equation}
			      \label{eq:very weak solution}
			      \begin{aligned}
				      \int_{0}^T  \int_{\Rd} u(t,x) \dRL \varphi(T-t,x)dxdt
				       & = \int_0^T \int_{\Rd} u^{m}(t,x) \Delta\varphi(T-t,x)  dx dt
				      \\
				       & \quad + M I_{1-\alpha}\varphi(T,0).
			      \end{aligned}
		      \end{equation}
		      for all $\varphi \in C^\infty_c ([0,T] \times \mathbb R^d)$. Furthermore, it satisfies \eqref{eq:initial Dirac} and \eqref{eq:profile mass}.
		\item
		      \label{it:slow-diffusion U near 0}
		      Behaviour near $z = 0$: We have $z^{d-1}(U^m)' \in C([0,+\infty))$ and
		      \begin{equation}
			      \label{eq:limit of flux at z=0}
			      - \lim_{z \to 0^+} |\partial B_1| z^{d-1}\frac{dU^m}{dz}(z) = \frac {M} {\Gamma(1-\alpha)}.
		      \end{equation}
		      Furthermore, we have that
		      \begin{equation}
			      \begin{aligned}
				      \label{eq:profile at z=0}
				      \lim_{z \to 0^+}\frac{U_{\alpha,m,M}^m(z)}{\mathcal V_{\alpha,m}^m(z)} =
				      \begin{cases}
					      \displaystyle\frac{M}{|\partial B_1| (d-2)\Gamma(1-\alpha)} & \text{if } d\ge 3,
					      \vspace{2pt}                                                                     \\
					      \displaystyle\frac{M}{|\partial B_1| \Gamma(1-\alpha)}      & \text{if } d = 2,
					      \vspace{2pt}                                                                     \\
					      C(\alpha,m) M^{\frac{2m}{m+1}}  & \text{if } d= 1.
				      \end{cases}
			      \end{aligned}
		      \end{equation}
		      where $C(\alpha, m) >0$ is a constant that depends only on the reference profile but is not explicit, and we define the Newtonian potential without constant
		      \[
			      \mathcal V_{\alpha,m}^m(z) \coloneqq
			      \begin{dcases}
				      z^{-(d-2) } & \text{if } d \ge 3 \\
				      |\log {z}|                           & \text{if } d=2,    \\
				      1                                    & \text{if } d=1.
			      \end{dcases}
		      \]
	\end{enumerate}
\end{theorem}

By \eqref{eq:profile at z=0} and $\alpha-am+b(d-2)=0$,
for our self-similar solutions we have that for dimension $d \ge 3$ as $t \to \infty$, at each fixed $x\ne0$,
	\begin{equation*}
		t^\alpha u(t,x)^m \to \frac{M (-\Delta)^{-1} \delta_0(x)}{\Gamma(1-\alpha)}.
	\end{equation*}
	Notice that this explains the shape of the self-similar profile \eqref{eq:profile at z=0} close to $z = 0$. This is the exact ``internal behaviour" of self-similar profile maintaining the memory of initial Dirac.

We divide the proof into cases $m = 1$ in \Cref{sec:existence proof m=1}, $m > 1$ in \Cref{sec:existence proof m>1}, and $m \in (m_c,1)$ in \Cref{sec:existence proof m < 1}.

\begin{remark}
	Notice that due to \eqref{eq:profile at z=0} we see that for $d \ge 2$ we have $\mathcal V_{\alpha,m}^m$ is the Newtonian potential up to a constant, and $u^m$ has the same leading singularity near the origin.
	Therefore, by \eqref{eq:limit of flux at z=0}, the distributional Laplacian $\Delta u^m$ has the Dirac contribution $-M t^{-\alpha}\delta_0/\Gamma(1-\alpha)$, together with a spatially distributed part.
	This means that $\partial_t^\alpha$ always remembers the initial Dirac.
	This was known for $m = 1$.
	In bounded domains and $0 < \alpha < 1$ it was shown in \cite[Theorem 1.4]{bonforteTimeFractionalPorousMedium2024} that if $u_0 \in L^p \cap H^*$ for $p\ge1+m$ and $p>d/(2s)$, then $u^m \in L^{\infty}$ at positive times, for $m>0$, $m\ne1$, $d>2s$, and under the stated operator assumptions.
		For $s=1$ and $d>2$, the strict inequality $p>d/2$ ensures that the Dirichlet Laplacian in bounded domains $-\Delta_\Omega$ satisfies $(-\Delta_\Omega)^{-1} : L^p \to L^\infty$; this fails in general at $p=d/2$.
	The authors had already pointed out that the memory effect makes $u^m$ behave like $(-\Delta_\Omega)^{-1} u_0$. This is \emph{elliptic regularization at work}.
	Lastly, we recall that the classical Barenblatt for $\alpha = 1$ is bounded and smooth at $z = 0$.
\end{remark}

\begin{remark}
	The theory of the existence of the profile in the linear case $m = 1$ is already well understood, and we only revise it in \Cref{sec:m=1}.
	We believe that \eqref{eq:very weak solution} and \eqref{eq:limit of flux at z=0} are also new for $m=1$.
\end{remark}

\subsection{The linear-diffusion case $m = 1$}
In \Cref{sec:m=1} we recall the deduction the closed formula
\begin{equation}
	\label{eq:m=1 self-similar}
	u_{\alpha,1,M}(t,x) = M \mathcal F^{-1}[ \widehat Z(t,\cdot) ] (x), \quad \widehat Z(t,\xi) = E_\alpha(-|\xi|^2t^\alpha),
\end{equation}
where $E_\alpha$ is the Mittag-Leffler function. In dimension $d=1$, $Z$ can be written in terms of the Wright function, see \cite{MR1829592}.
Due to the explicit construction, it was proved in \cite{dier2020} for $\alpha \in (0,1)$
that for $z \ge 1 $ there exists constants such that
\begin{equation*}
	c \mathcal V_{\alpha,1} (z)
	\le
	{U_{\alpha,1,M}(z)}
	\le C\mathcal V_{\alpha,1}(z),
	\quad
	\mathcal V_{\alpha,1}(z) \coloneqq z^{\frac{d(\alpha-1)}{2-\alpha}} e^{-\sigma_\alpha z^{\frac 2{2-\alpha}}},
	\quad
	\sigma_\alpha \coloneqq (2-\alpha)(\alpha^\alpha/4)^{1/(2-\alpha)}.
\end{equation*}
Further asymptotic analysis can be found in \cite{deng2019}.
It is known that the fundamental solution describes the intermediate asymptotics of the solution with general initial datum $u_0$, see, e.g. \cite{cortazarHeatEquationMemory2021}. The uniqueness of weak solutions holds by duality.
\begin{theorem}
	\label{thm:m=1 uniqueness}
	Let $d \ge 1, \alpha \in (0,1)$, $m = 1$, and $M > 0$. Then, there exists a unique solution to \eqref{eq:very weak solution} in $L^\infty(0,T;L^1(\Rd))$ for every $T>0$.
	Hence, there is at most one solution to \eqref{eq:profile equation} such that \eqref{eq:profile mass}.
\end{theorem}

\subsection{The slow-diffusion regime $m > 1$}
In the slow-diffusion case, we can a complete description of the self-similar profiles. Similarly to the case $\alpha = 1$ they are compactly supported and of power-type at the free boundary.
\begin{theorem}[Slow-diffusion regime]
	\label{thm:slow-diffusion}
	Let $d \ge 1$, $\alpha \in (0,1)$, and $m > 1$.
	Then, the profiles $\mathcal U_{\alpha,m}$ and $U_{\alpha,m,M}$ in \Cref{thm:existence} have the following properties
	\begin{enumerate}
		\item
		      $U_{\alpha,m,M}$ is the unique non-negative point-wise solution to \eqref{eq:profile equation} such that \eqref{eq:profile mass}.
		\item The solutions are ordered according to their mass and for each $z > 0$
		      \[
			      U_{\alpha,m,M}(z) \to +\infty  \quad \text{ everywhere as } M \to +\infty .
		      \]
		\item
		      \label{it:slow-diffusion U compactly supported}
		      $\mathcal U_{\alpha,m}$ is continuous, non-increasing, and supported in $[0,1]$.
		\item
		      Behavior at the free boundary:
		      \begin{equation}
			      \label{eq:slow sharp constant at free boundary}
			      \lim_{z \to 1^-} \frac{\mathcal U_{\alpha,m}(z)}{(1-z)_+^{\frac {2-\alpha }{m-1}}} = \left( \frac{b_{\alpha,m}^\alpha \Gamma(1 + \tfrac{2-\alpha}{m-1} )}{\Gamma(3-\alpha+\tfrac{2-\alpha}{m-1})} \right)^{\frac{1}{m-1}}.
		      \end{equation}
		\item Classical limit as $\alpha \to 1^-$: for $z > 0$ we have
		      \begin{equation}
			      \label{eq:m>1 limit alpha->1}
			      \lim_{\alpha \to 1^-}\mathcal U_{\alpha,m}(z) =  \left( \frac{(m-1)b_{1,m}}{2m} (1-z^2)_+ \right)^{\frac 1 {m-1}} \eqqcolon \mathcal U_{m}(z).
		      \end{equation}
		\item
		      \label{thm:mesa}
		      The mesa limit: for any $t > 0$ and $x \in \mathbb R^d$ with $|x|\notin\{0,R_M\}$, where $R_M=(M/|B_1|)^{1/d}$, we have that
		      \begin{equation}
			      \lim_{m \to \infty} u_{\alpha,m,M}(t,x) = \chi_{B_{R_M}}(x) \eqqcolon u_{\infty,M}(x).
		      \end{equation}
		      The convergence also holds in $L^1(\mathbb R^d)$ for every fixed $t>0$.
	\end{enumerate}
\end{theorem}

The proof of this result can be found in \Cref{sec:existence proof m>1}.
Notice that $\mathcal U_{m}$ is the famous Barenblatt solution; see, e.g. \cite[(17.28)-(17.30)]{Vazquez2006}.

\begin{remark}
	Notice that the mesa limit does not depend on $\alpha$. Moreover, it has a bounded $L^1$ representative and does not change with time.
\end{remark}

\begin{figure}[h!]
	\centering
	\includegraphics[width=0.45\textwidth]{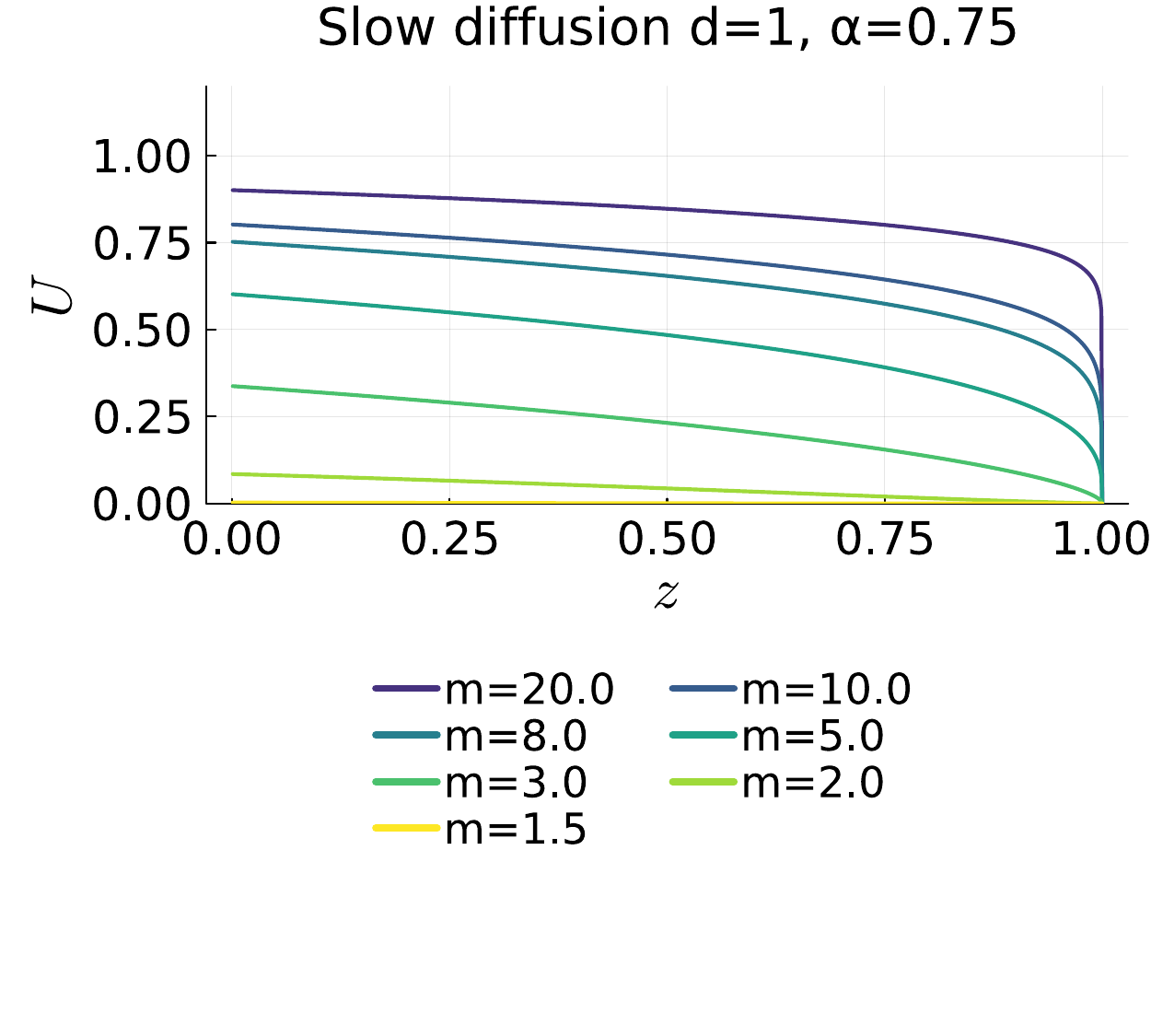}
	\includegraphics[width=0.45\textwidth]{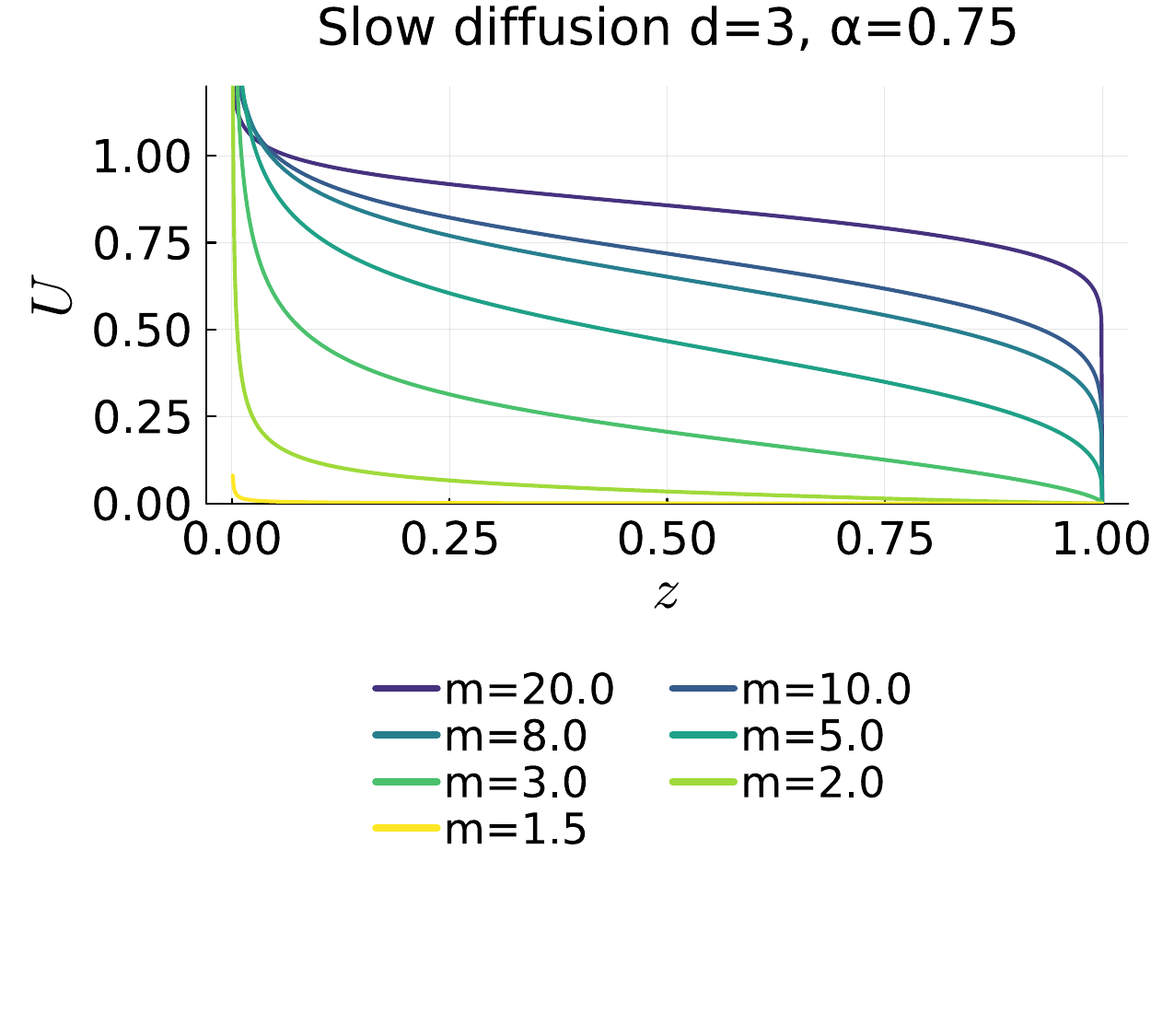}
	\caption{Slow-diffusion profiles for different values of $m$}
	\label{fig:mesa}
\end{figure}

\begin{figure}[h!]
	\centering
	\includegraphics[width=0.45\textwidth]{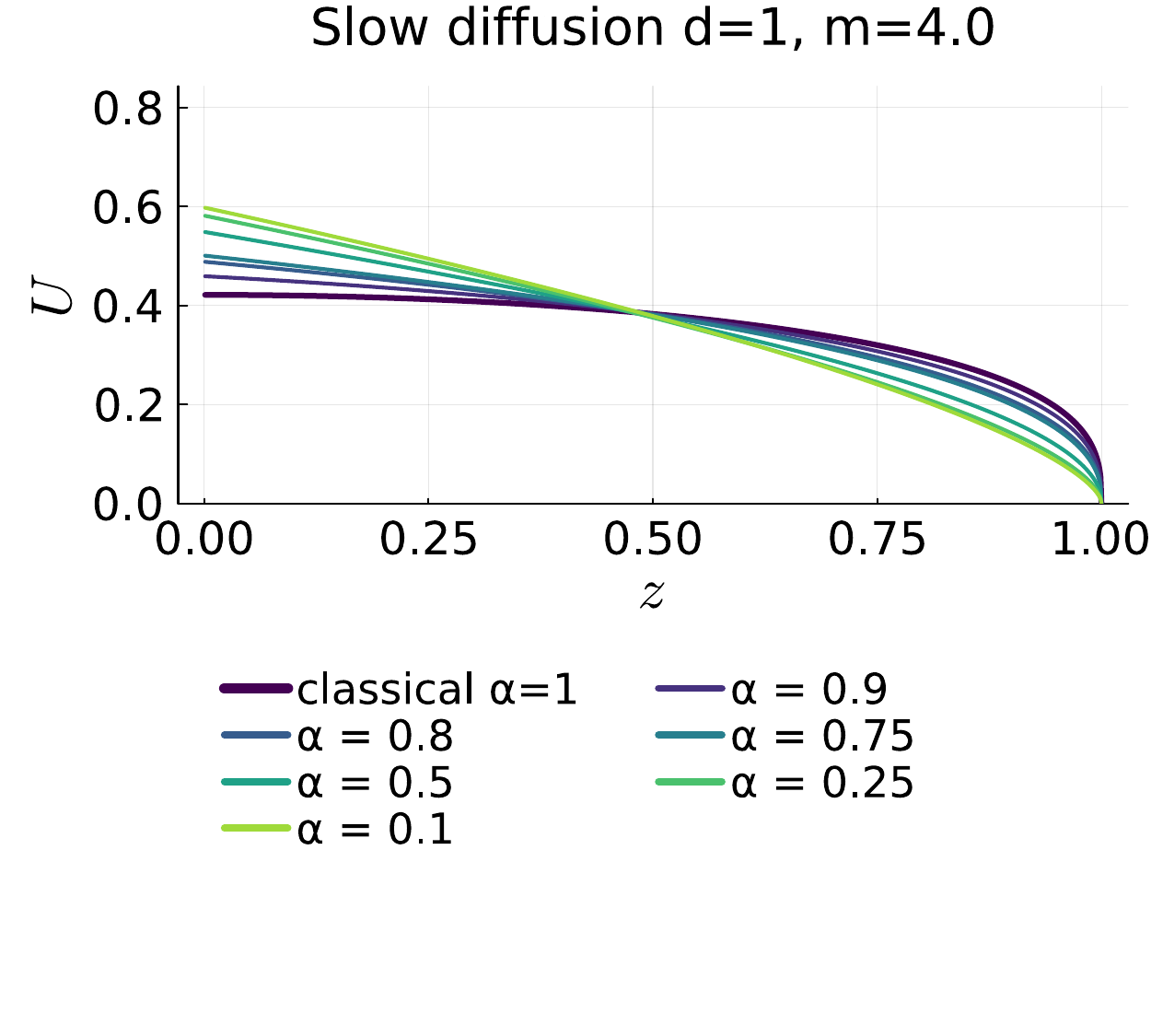}
	\includegraphics[width=0.45\textwidth]{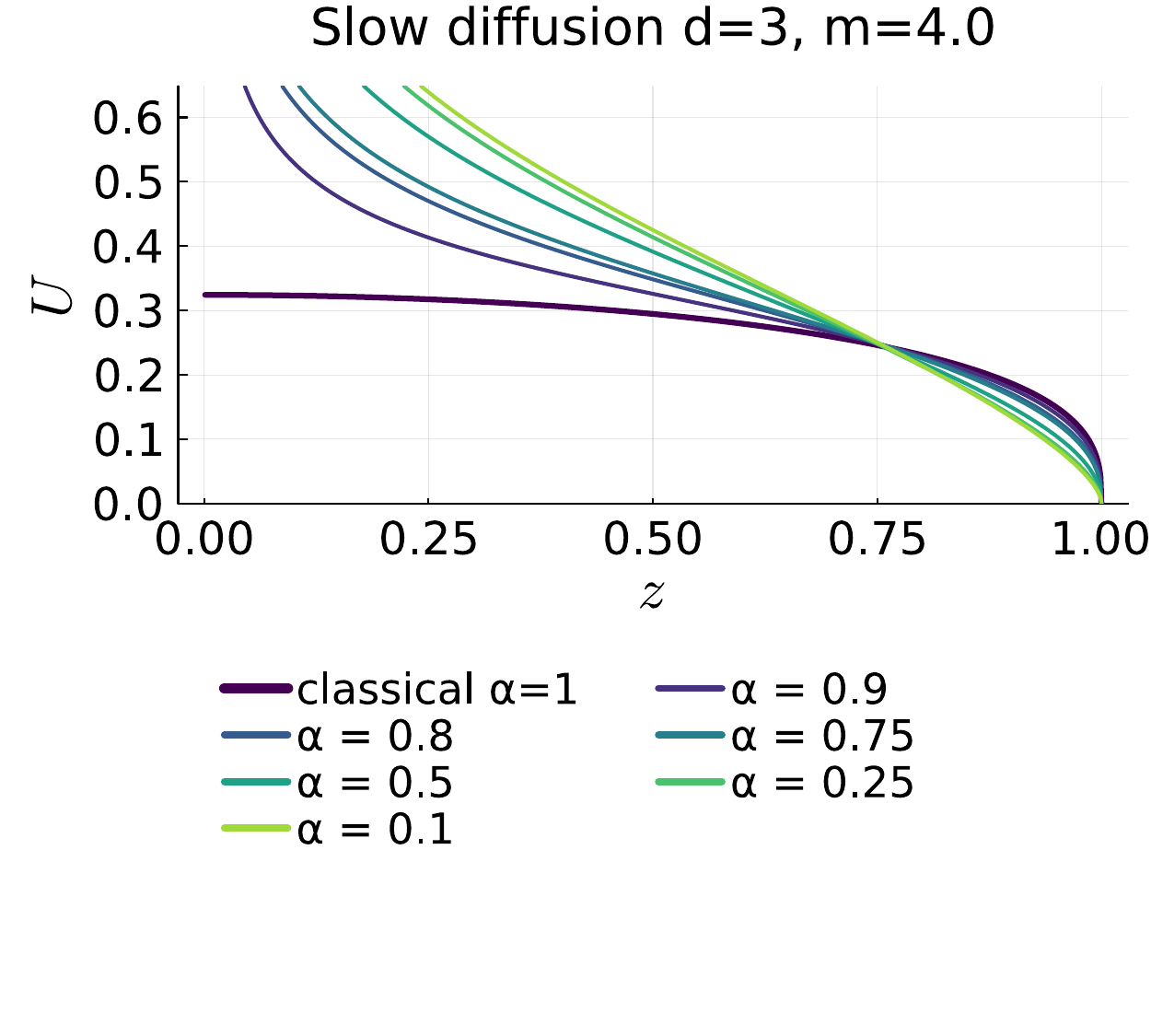}

	\caption{Slow-diffusion profiles for different values of $\alpha$.}
	\label{fig:limit alpha->1}
\end{figure}

\subsection{The mildly-fast-diffusion regime $m \in (m_c, 1)$}
We study the mildly-fast-diffusion regime $m_c<m<1$, where the critical value $m_c$ is the same as in the classical case $\alpha = 1$.
We recall that for $\alpha=1$ the self-similar profiles are $U_{1,m}(z) = c_{1,m}^* (z^2 + c)^{-\frac 1 {1-m}}$ where $c_{1,m}^*$ depends only on $m$ and $d$, and $c>0$ is a free constant.
These solutions have finite mass.
As $c \to 0^+$ we recover the so-called \emph{very-singular solution} (VSS), which has infinite mass. We will show that for $\alpha \in (0,1)$ there is the \emph{time-fractional very singular solution} given by
\begin{equation}\label{eq:vss}
	U_{\alpha,m}^*(z) = c_{\alpha,m}^* z^{-\frac 2 {1-m}}, \quad  c_{\alpha,m}^* = \left( \frac{(1-m)b_{\alpha,m}}{2m} \frac{\Gamma(\tfrac{\alpha}{1-m})}{\Gamma(1-\alpha + \frac{\alpha}{1-m})}  \right)^{\frac{-1}{1-m}} > 0.
\end{equation}
This construction is justified in \Cref{lem:vss}. Notice that this function gives an exact self-similar solution in separated variables
\begin{equation*}
	u_{\alpha,m}^*(t,x) \coloneqq t^{-bd} U_{\alpha,m}^*(t^{-b} x) = c^*_{\alpha,m} \left(\frac{t^{{\alpha}}}{|x|^{2}}\right)^{\frac 1 {1-m}}.
\end{equation*}
This solution starts from a singularity point at $t=0$ and increases in time everywhere for all $x\ne 0$ for all $t>0$.
The convention \emph{very-singular solution} is recent, and some previous literature has referred to this solution as \emph{infinite point-source solutions} (cf. \cite{Vazquez2006}).
We prove the existence of finite-mass self-similar solutions that form an \emph{asymptotic fan} around this very singular solution.

\begin{theorem}[Fast-diffusion regime, mildly-singular range]
	\label{thm:fast-diffusion}
	Assume that $d \ge 1$, $\alpha \in (0,1)$, $m \in (m_c, 1)$. Then the self-similar profiles from \Cref{thm:existence} are such that
	\begin{enumerate}
		\item $\mathcal U_{\alpha,m}$ and $U_{\alpha,m,M}$ are decreasing and strictly positive for $z \in (0,+\infty)$.

		\item The tail behavior is given by the asymptotic expansion
		      \begin{equation}
			      \label{eq:fast-diffusion U tail}
			      \begin{aligned}
				      \mathcal U_{\alpha,m}(z) & = U^*(z) ( 1 - z^{-\gamma^*} + \mathcal R_{\alpha,m}(z)),
				      \\
				      U_{\alpha,m,M}(z)        & = U^*(z) ( 1 - (Lz)^{-\gamma^*} + \mathcal R_{\alpha,m}(Lz))
			      \end{aligned}
		      \end{equation}
		      where $L = A^{-\frac{m-1}{2}}$ is given by \eqref{eq:rescaling profile to correct mass}, $\mathcal R_{\alpha,m}(z) = o(z^{-\gamma^*})$, and the exact value of $\gamma^*$ is implicitly given by \eqref{eq:gamma^star}.
		\item The solutions are ordered with respect to $M$, and furthermore
		      \begin{align*}
			       & U_{\alpha,m,M}(z) \nearrow U_{\alpha,m}^*(z) \text{ as } M \to \infty
			      \quad \text{and} \quad
			      U_{\alpha,m,M}(z) \searrow 0 \text{ as } M \to 0.
		      \end{align*}
		\item Classical limit as $\alpha \to 1^-$: We have that
		      \begin{equation}
			      \label{eq:m<1 alpha->1}
			      \lim_{\alpha \to 1^-}\mathcal U_{\alpha,m} (z) = c^*_{1,m} (1-m+z^2)_+ ^{\frac 1 {m-1}} \eqqcolon \mathcal U_{m}(z).
		      \end{equation}
	\end{enumerate}
\end{theorem}

\subsection{Continuity in $m$ at $m=1$.}
Our last result is the matching of the three behaviors. Due to the criteria by which we have chosen the ``canonical profile'' $\mathcal U_{\alpha,m}$ we do not expect continuity w.r.t. $m$ as $m$ crosses the critical value $m = 1$ (see \eqref{eq:m>1 limit alpha->1} and \eqref{eq:m<1 alpha->1}).
Furthermore, it is easy to notice that $\mathcal U_{1,m} \to +\infty$ as $m \to 1^-$ and $\mathcal U_{1,m} \to 0$ as $m \to 1^+$ at each fixed $z>0$.
However, the correct continuity in $m$ holds for the solutions of given mass, as shown by the following result.
\begin{theorem}
	\label{thm:limit m->1}
	Let $d \ge 1$, $\alpha \in (0,1)$, and $M > 0$. Then, for any $z > 0$ we have that
	\begin{equation*}
		\lim_{m \to 1} U_{\alpha,m,M}(z) = U_{\alpha,1,M}(z).
	\end{equation*}
\end{theorem}

\section{Weak formulation}

In this section we discuss the weak formulation, and provide a lemma that allows to characterize self-similar solutions as weak solutions based on the properties of the profile.

\subsection{Weak formulation for smooth solutions}
We begin by showing that the weak formulation is reasonable for smooth solutions.
	For $f,g\in C^1([0,T])$, by Fubini, we deduce that
\begin{align*}
	\Gamma(1-\alpha) \int_0^T [I_{1-\alpha} f(t)] \, g(T-t) dt
	 & = \int_0^T \int_0^t (t-s)^{-\alpha} f(s) g(T-t) ds dt
	\\
	 & = \int_0^T \int_s^{T} (t-s)^{-\alpha} g(T-t) dt \, f(s) ds
	\\
	 & = \int_0^T \int_0^{T-s} (T-\sigma -s)^{-\alpha} g(\sigma) d\sigma \, f(s)  ds
	\\
	 & = \Gamma(1-\alpha) \int_0^T f(s) I_{1-\alpha} g (T-s) ds.
\end{align*}
This means that
\begin{equation*}
	\int_0^{T} \partial_t^\alpha f(t) g(T-t) dt = \int_0^T f(t) \dRL g (T-t)dt + f(T) I_{1-\alpha} g(0) - f(0) I_{1-\alpha} g(T).
\end{equation*}
Notice that if $g \in C([0,T])$ then $I_{1-\alpha} g(0) = 0$ for $\alpha \in (0,1)$.
For $\alpha = 1$ we have $I_{1-\alpha} g = g$.
We write therefore that
\begin{align*}
	\int_{0}^T \int_{\Rd}  u(t,x) \dRL  \varphi (T-t,x)dxdt
	 &
	=
	\int_\Rd u(0,x) I_ {1-\alpha} \varphi(T,x) dx - \int_\Rd u(T,x) I_{1-\alpha}\varphi(0,x) dx \\
	 & \quad +
	\int_0^{T} \int_{\Rd} u^{m} (t,x) \Delta \varphi (T- t,x) dxdt.
\end{align*}
Since $\varphi$ is bounded and $0<\alpha<1$, we have $I_{1-\alpha}\varphi (0,x) = 0$ and therefore
\begin{equation*}
	\begin{aligned}
		\int_{0}^T \int_{\Rd}  u(t,x) \dRL  \varphi (T-t,x)dxdt
		 &
		=
		\int_\Rd u(0,x) I_ {1-\alpha} \varphi(T,x) dx
		\\
		 & \quad +\int_0^{T} \int_{\Rd} u^{m} (t,x) \Delta \varphi (T- t,x) dxdt.
	\end{aligned}
\end{equation*}
This integration by parts holds for all $\varphi \in C_c^\infty ([0,T] \times \mathbb R^d)$ since the term $\varphi(0,x)t^{-\alpha}/\Gamma(1-\alpha)$ in $\dRL\varphi$ is integrable in time.

\subsection{Weak formulation for finite-mass self-similar solution}
Now we provide a general lemma to show that a self-similar solution is a weak solution in the sense of \eqref{eq:very weak solution}.
\begin{lemma}
	\label{lem:sss is weak solution}
	Let $d \ge 1$, $m>m_c$, $\alpha \in (0,1)$, $u$ be given by \eqref{eq:ssf} with the exponents in \eqref{eq:self-similar exponents}.
	Assume that $U$ is continuous on $(0,\infty)$, non-increasing, and satisfies \eqref{eq:limit of flux at z=0}, and that $u$
	satisfies \eqref{eq:mass conservation} for all $t > 0$,
	it is a classical solution to \eqref{eq:main} in
	\begin{equation*}
		S \coloneqq \{(t,x) : t > 0, x \ne 0, u(t,x) > 0\},
	\end{equation*}
	and that $\nabla u^{m}$ is continuous at $\partial \operatorname{supp}(u)$.
	Then $u$ satisfies \eqref{eq:very weak solution}.
\end{lemma}

\begin{proof}
	Let $R \le \infty$ be such that the support of $U$ is $[0,R]$.
	\textbf{First integration by parts.}
	Let us take $\varphi \in C_c^\infty([0,T] \times \mathbb R^d)$ and for $\delta \in [0,1)$ let us define
	\[
		\Omega_\delta(t) =
		\begin{cases}
			B_{(1-\delta)Rt^b} & \text{if } R < \infty , \\
			\mathbb R^d        & \text{if } R = \infty.
		\end{cases}
	\]
	Notice that if $x \in \Omega_\delta(t) \setminus B_\varepsilon $ and $t > 0$ then $(t,x) \in S$.
	Since $u$ is a classical solution in $S$, we have
	\begin{equation*}
		\int_{0}^T \int_{\Omega_\delta(t) \setminus B_\varepsilon} \partial_t^\alpha u(t,x)  \varphi (T-t,x)dxdt
		=
		\int_0^{T} \int_{\Omega_\delta(t) \setminus B_\varepsilon} \Delta u^{m} (t,x) \varphi (T- t,x) dxdt.
	\end{equation*}
	To integrate by parts the Caputo derivative, we write
	\begin{equation*}
		\int_{0}^T \int_{\Omega_\delta(t) \setminus B_\varepsilon} \partial_t^\alpha u(t,x)  \varphi (T-t,x)dxdt
		= \int_{\Omega_\delta(T) \setminus B_\varepsilon} \int_{A_\delta(x)}^T \partial_t^\alpha u(t,x)  \varphi (T-t,x)dtdx.
	\end{equation*}
	Here $A_\delta(x)=(|x|/((1-\delta)R))^{1/b}$ if $R<\infty$, and $A_\delta(x)=0$ if $R=\infty$. For fixed $x$ we put $A=A_\delta(x)$, $f(t)=u(t,x)$ and $g(s)=\varphi(s,x)$.
	By monotonicity and finite mass, we have that $z^d U(z)\to0$ at infinity and, for $L>0$, $\int_L^\infty z^d|U'(z)|\,dz=L^dU(L)+d\int_L^\infty z^{d-1}U(z)\,dz$. Thus, $f(t)=|x|^{-d}z^dU(z)$, $z=t^{-b}|x|$, is absolutely continuous with $f(0)=0$ and $\|f'\|_{L^1(0,T)}$ is bounded on compact annuli.
	We repeat the integration by parts of the Caputo derivative
	\begin{align*}
		\int_A^T \partial_t^\alpha f(t) g(T-t) dt & = \int_0^T \partial_t f(t) I_{1-\alpha} g (T-t) dt - \int_0^A \partial_t f (t) R(t) dt,
	\end{align*}
	where
	\[
		R(t) \coloneqq \int_{T-A}^{T-t} \frac{(T-t-\sigma)^{-\alpha}}{\Gamma(1-\alpha)} g(\sigma) d \sigma.
	\]
	Integrating by parts, we recover
	\begin{align*}
		\int_A^T \partial_t^\alpha f(t) g(T-t) dt
		 & = \int_0^T  f(t) \dRL g (T-t) dt  - f(0) I_{1-\alpha} g(T)                                                                                       \\
		 & \quad + \int_0^A f (t) \partial_t R(t) dt - f(A) R(A) + f(0) R(0).
	\end{align*}
	As $\delta\to0$, $A_\delta(x)$ converges to $A_0(x)$, and $f'=0$ before that time. Also, the remainder $\int_0^A f'(t)R(t)\,dt$ therefore goes to zero by the dominated convergence theorem (since $|R(t)|\le T^{1-\alpha}\|g\|_\infty/\Gamma(2-\alpha)$).
	Since $u(0,x)= 0$ in $\Rd \setminus B_\varepsilon$ we obtain as $\delta \to 0$ that
	\begin{align*}
		\int_0^T \int_{\Omega_\delta(t) \setminus B_\varepsilon} \partial_t^\alpha u(t,x)  \varphi (T-t,x)dxdt
		 & \overset{\delta \to 0}\longrightarrow
		\int_0^T \int_{\Omega_0(t) \setminus B_\varepsilon} u(t,x) \dRL \varphi (T-t,x)dxdt
		\\
		 & =
		\int_0^T \int_{\Rd \setminus B_\varepsilon} u(t,x) \dRL  \varphi (T-t,x)dxdt.
	\end{align*}
	On the other hand, we have that
	For $R<\infty$, let $E_{\delta,\varepsilon}=\{t\in(0,T):(1-\delta)Rt^b>\varepsilon\}$. The boundary integrals below are restricted to this set. When $R=\infty$, take $E_{\delta,\varepsilon}=(0,T)$ and omit the outer boundary integral.
	\begin{align*}
		\int_0^{T} \int_{\Omega_\delta(t) \setminus B_\varepsilon} \Delta u^{m} (t,x) \varphi (T- t,x) dxdt
		 & = - \int_0^{T} \int_{\Omega_\delta(t) \setminus B_\varepsilon} \nabla u^{m} (t,x) \nabla \varphi (T- t,x) dxdt
		\\
		 & \quad + \int_{E_{\delta,\varepsilon}} \int_{\partial B_\varepsilon} \varphi(T-t,x) \nabla u^{m}(t,x) \cdot {\frac{-x}{|x|}}
		\\
		 & \quad + \int_{E_{\delta,\varepsilon}} \int_{\partial B_{(1-\delta)Rt^b}} \varphi(T-t,x) \nabla u^{m}(t,x) \cdot {\frac{x}{|x|}}.
	\end{align*}
	For finite $R$, these boundary terms have $t>(\varepsilon/R)^{1/b}$. Continuity of $\nabla u^m$ at the free boundary therefore makes the outer term tend to zero as $\delta\to0$. The inner term may then be extended to $(0,T)$, since its integrand vanishes before $t=(\varepsilon/R)^{1/b}$. Thus
	\begin{align*}
		\int_0^{T} \int_{\Omega_\delta(t) \setminus B_\varepsilon} \Delta u^{m} (t,x) \varphi (T- t,x) dxdt
		 & \overset{\delta \to 0} \longrightarrow
		- \int_0^{T} \int_{\mathbb R^d \setminus B_\varepsilon} \nabla u^{m} (t,x) \nabla \varphi (T- t,x) dxdt
		\\
		 & \quad + \int_0^T \int_{\partial B_\varepsilon} \varphi(T-t,x) \nabla u^{m}(t,x) \cdot {\frac{-x}{|x|}}.
	\end{align*}
	We conclude that
	\begin{equation}
		\label{eq:weak formulation 1}
		\begin{aligned}
			\int_{0}^T \int_{\Rd \setminus B_\varepsilon} u(t,x) \dRL \varphi (T-t,x)dxdt
			 & =
			- \int_0^{T} \int_{\mathbb R^d \setminus B_\varepsilon} \nabla u^{m} (t,x) \nabla \varphi (T- t,x) dxdt
			\\
			 & \quad
			+ \int_0^T \int_{\partial B_\varepsilon} \varphi(T-t,x) \nabla u^{m}(t,x) \cdot {\frac{-x}{|x|}}.
		\end{aligned}
	\end{equation}

	\textbf{Limit as $\varepsilon \to 0$.}
	Since $u \in L^\infty(0,T; L^1 (\mathbb R^{d}))$ and $|(\dRL\varphi)(T-t,x)|\le C(1+(T-t)^{-\alpha})$ we have that
	\[
		\int_{0}^T \int_{\Rd \setminus B_\varepsilon} u(t,x) \dRL \varphi (T-t,x)dxdt
		\overset{\varepsilon \to 0} \longrightarrow
		\int_{0}^T \int_{\Rd} u(t,x) \dRL \varphi (T-t,x)dxdt.
	\]
	We notice that $\nabla u^m$ is not a nice function in $x = 0$. We integrate by parts again in space
	\begin{align*}
		\int_0^T \int_{\Rd \setminus B_\varepsilon} u(t,x) \dRL \varphi (T-t,x)dxdt
		 & =
		\int_0^{T} \int_{\mathbb R^d \setminus B_\varepsilon} u^{m} (t,x) \Delta \varphi (T- t,x) dxdt
		\\
		 & \quad
		+ \int_0 ^T \int_{\partial B_\varepsilon} \varphi(T-t,x) \nabla u^{m}(t,x) \cdot {\frac{-x}{|x|}}
		\\
		 & \quad
		-
		\int_0^T \int_{\partial B_\varepsilon} u^{m}(t,x) \nabla \varphi(T-t,x) \cdot {\frac{-x}{|x|}}.
	\end{align*}
	By radial symmetry and the divergence theorem, the absolute value of the last term is at most
	\[
		\|\Delta\varphi\|_\infty\int_0^T\int_{B_\varepsilon}u^m\,dx\,dt,
	\]
	which goes to zero. Indeed, $u^m\le1+u$ if $m\le1$. If $m>1$, \eqref{eq:limit of flux at z=0} gives $|(U^m)'(z)|\le C z^{1-d}$ near zero, hence $U^m$ is locally integrable with weight $z^{d-1}$. In particular, for $d\ge3$, integration gives $U^m(z)=O(z^{2-d})$ and $z^{d-1}U^m(z)=O(z)$. At infinity $U$ is bounded and has finite mass, so $\int_{\Rd}u^m(t,x)\,dx=Ct^{2b-\alpha}$ is integrable in time.
	For the second-to-last term, we again use \eqref{eq:limit of flux at z=0} in its weaker form.
	Again we write that
	\begin{align*}
		\int_{\partial B_\varepsilon} \varphi(T-t,x) \nabla u^{m}(t,x) \cdot \frac{-x}{|x|}
		 & = G(t^{-b}\varepsilon) t^{-\alpha} |\partial B_\varepsilon|^{-1} \int_{\partial B_\varepsilon} \varphi(T-t, x) dS_x
		,
	\end{align*}
	where
	$G(z) = -|\partial B_1|\frac{dU^{m}}{dz} z^{d-1}$.
	For $t\ge\varepsilon^{1/b}$ we have that $G(t^{-b}\varepsilon)$ is bounded. Also, for $q=(1-\alpha)/b$, the integral of $t^{-\alpha}G(t^{-b}\varepsilon)$ over $(0,\varepsilon^{1/b})$ is $\varepsilon^q b^{-1}\int_1^\infty z^{-1-q}G(z)\,dz$. This last $z$-integral is finite by integration by parts, since $U(z)\le C z^{-d}$ and $d-2-q-dm=-1/b<0$.
	Now we use the smoothness of $\varphi$. Let $\eta(t) = \varphi(t,0)$ and we have that
	\begin{equation*}
		\sup_{t\in [0,T]} \left| \eta(t) - |\partial B_\varepsilon|^{-1} \int_{\partial B_\varepsilon} \varphi(t,x) \right| \le \|\nabla \varphi\|_{L^\infty} \varepsilon.
	\end{equation*}
	Hence, we have that
	\begin{equation*}
		\left| \int_0^T \int_{\partial B_\varepsilon} \varphi(T-t,x) \nabla u^{m}(t,x) \cdot \frac{-x}{|x|} - \int_0^T \eta(T-t) t^{-\alpha} G(t^{-b} \varepsilon) dt \right|  \le C \varepsilon.
	\end{equation*}
	As $\varepsilon \to 0$ we use \eqref{eq:limit of flux at z=0} to deduce that
	\[
		\int_0^T \int_{\partial B_\varepsilon} \varphi(T-t,x) \nabla u^{m}(t,x) \cdot \frac{-x}{|x|} \overset{\varepsilon \to 0} \longrightarrow M I_{1-\alpha} \varphi(T,0).
	\]
	In the end, we conclude \eqref{eq:very weak solution}.
\end{proof}

	\paragraph{Initial trace.}
	For every finite-mass profile constructed here, $a=db$ and $b>0$. Hence, for $\psi\in C_c^\infty(\Rd)$, the change of variables $x=t^b y$ gives
	\[
		\int_{\Rd}u(t,x)\psi(x)\,dx
		=\int_{\Rd}U(|y|)\psi(t^b y)\,dy
		\longrightarrow M\psi(0)\qquad\text{as }t\to0^+.
	\]
	Dominated convergence applies with majorant $\|\psi\|_\infty U(|y|)\in L^1(\Rd)$, proving \eqref{eq:initial Dirac}.

\section{Equation for the self-similar profile}
In this section, we will derive the integral equation that represents the self-similar solution and prove several of its properties.

\subsection{Scaling properties of \eqref{eq:main}}
\label{sec:scaling}
First, let us obtain the scaling factors $a$ and $b$ through scaling arguments at the PDE level.
This equation has the following scaling. If $u(t,x)$ is a solution of a given mass $M$, then
\begin{equation*}
	\widetilde u(t,x) = A u(Tt, Lx ),
\end{equation*}
is also a solution if the three non-negative scaling constants are related by
\begin{equation}
	\label{eq:dimensional scaling}
	A^{m-1} L^2 = T^\alpha.
\end{equation}
Hence, if $U$ is a profile, then
\[
	\widetilde U (z) = A T^{-bd} U(T^{-b} L z),
\]
is also a self-similar profile of the same equation. We point out that
\[
	\int_\Rd \widetilde u(t,x) dx = A L^{-d} \int_\Rd u(Tt, y) dy,
\]
so the scaling of the masses is
\[
	\widetilde M = A L^{-d} M.
\]
Hence, once we show the existence of a profile of a certain mass, we will immediately have the existence of profiles of all masses.

To obtain a self-similar scaling, we set $\widetilde U = U$. We must therefore have $A L^{-d} = AT^{-bd} = T^{-b}L = 1$ so $A = T^{bd}$ and $L = T^b$. Lastly, matching the scaling of equation \eqref{eq:dimensional scaling}, we deduce that $bd(m-1) + 2b = \alpha$, hence the correct scaling of dimensions \eqref{eq:self-similar exponents}. On the other hand, if we aim to change mass, we set $T = 1$, and we have that
\begin{equation}
	\label{eq:scaling solutions}
	\widetilde U (z) = A U(A^{-({m-1})/2} z), \quad L = A^{-({m-1})/2}
\end{equation}
is also a solution for any $A > 0$.
We are ready to discuss the existence of a suitable profile $U\ge 0$ with finite mass and compact support.

\subsection{Derivation of \eqref{eq:profile equation}}
Let us obtain the equation for the profile by inserting formula \eqref{eq:ssf} into equation \eqref{eq:main}.
We follow the same approach as \cite{CaballeroOkrasinska-PlociniczakPlociniczakSadarangani2025}, but now in general dimension $d \ge 1$ and keeping track that reasoning is valid also for $m \in (m_c,1]$.

\paragraph{Scaling the Caputo derivative.}
For $x \ne 0$ we have $u(0,x) = 0$ hence we can move the derivative outside the integral and write \eqref{eq:Caputo}
\begin{align*}
	\partial_t^\alpha u(t,x)
	 & = I_{1-\alpha} \partial_t u (t,x)                                                                                                           = \partial_t I_{1-\alpha} u (t,x) %
	= t^{-\alpha-a} \left( \mathsf A F(z) - b z \frac{dF}{dz}  \right)\Bigg|_{z = t^{-b}|x|}
\end{align*}
where $\mathsf A = 1-\alpha-a$ and
\[
	F(z) = \frac 1  {\Gamma(1-\alpha)}\int_0^1 (1-\sigma)^{-\alpha} \sigma^{-a} U(\sigma^{-b} z) d\sigma.
\]
Notice that, even though the integrand contains a singular factor $\sigma^{-a}$ with possibly $a \ge 1$,  this integral is well-defined if $|x| > 0$ and $U$ decays sufficiently fast at infinity.

\paragraph{Scaling the diffusion.} The spatial part, on the other hand, under the scaling behaves as follows
\[
	\Delta u^{m} (t,x) = t^{-2b-am} \left[ z^{1-d} \frac{d}{dz}\left(z^{d-1} \frac{dU^{m}}{dz}\right) \right] \Bigg|_{z=t^{-b} |x|}.
\]

\paragraph{An equation in self-similar variable.}
At this stage, the self-similar exponents \eqref{eq:self-similar exponents} can be deduced again.
Collecting the above scaling relations, we arrive at the integro-differential equation
\begin{equation}
	\label{eq:profile equation second order}
	z^{1-d} \frac{d}{dz} \left(z^{d-1} \frac{dU^{m}}{dz} \right) = \mathsf A F(z) - b z \frac{dF}{dz}.
\end{equation}
At this stage, this is a formal derivation. For the profiles we construct, we will show that this equation holds pointwise in the interior of the support and we will show that $u$ is a weak solution of \eqref{eq:main}.

\paragraph{An integral equation.} If we have that
\begin{equation}
	\lim_{z \to \infty} z^{d-1}\frac{dU^{m}}{dz}(z) = 0,
	\quad
	\lim_{z\to \infty} z^d F(z) = 0,
\end{equation}
then we have
\begin{align*}
	-z^{d-1} \frac{dU^{m}}{dz} (z) = \mathsf A \int_z^{+\infty} F(\sigma) \sigma^{d-1}d\sigma - b \int_{z}^{+\infty} \sigma^d \frac{dF}{dz}(\sigma) d \sigma.
\end{align*}
Integrating by parts and noticing that $\mathsf A + db = 1-\alpha$ we have that
\begin{equation}
	\label{eq:equation for dUm/dz}
	-z^{d-1} \frac{dU^{m}}{dz} (z) =  \int_z^{+\infty} ({1-\alpha}) F(\sigma) \sigma^{d-1} d\sigma + b z^{d} F(z).
\end{equation}

\begin{remark}
	\label{rem:decreasing}
	If a solution of \eqref{eq:equation for dUm/dz}
	is $U \ge 0$,  then $U$ is non-increasing.
\end{remark}
Lastly assuming that $U(+\infty) = 0$ we have that
\begin{equation}
	\label{eq:implicit equation last formula with F}
	\begin{aligned}
		U(z)^{m} & = \int_z^{+\infty} \sigma^{1-d} \int_\sigma^R ({1-\alpha}) \xi^{d-1} F(\xi) d \xi d \sigma + b \int_{z}^R \xi F(\xi)d\xi \\
		         & = \int_z^{+\infty} \Bigg( ({1-\alpha}) \int_z^\xi \sigma^{1-d} \xi^{d-1} d \sigma + b \xi \Bigg) F(\xi) d \xi            \\
		         & = \int_z^{+\infty} \Bigg( ({1-\alpha}) \xi \int_{z/\xi}^1 \eta^{1-d} d \eta + b \xi \Bigg) F(\xi) d \xi                  \\
		         & = \int_z^{+\infty} \Bigg(({1-\alpha}) H_d(z/\xi) + b \Bigg)\xi F(\xi) d \xi,
	\end{aligned}
\end{equation}
where for $s \in (0,1)$ we have the Newtonian potentials (up to a constant)
\[
	H_d(s) = \int_{s}^1 \tau^{1-d} d\tau =
	\begin{cases}
		1-s                   & \text{if } d=1,     \\
		\log (1/s)            & \text{if } d = 2,   \\
		\frac{s^{2-d}-1}{d-2} & \text{if } d \ge 3.
	\end{cases}
\]
\begin{remark}
	Notice that in this last integration step it is not easy to integrate in $[0,z]$ and write $-(U(0)^{m} - U(z)^{m})$, because instead of $H_d$ we would have $\int_0^s \tau^{1-d}$.
\end{remark}
Taking into account that the solution has a compact support $[0,R]$, the lower limit of the integral is actually larger than $0$, therefore
\[
	U(z)^{m} = \frac{1}{\Gamma(1-\alpha)}\int_z^{+\infty} \left(({1-\alpha}) H_d(z/\xi) + b  \right) \xi\left( \int_{0}^1 (1-\sigma)^{-\alpha} \sigma^{-a} U(\sigma^{-b} \xi) d \sigma \right) d \xi.
\]
By substituting $\rho = \sigma^{-b}\xi$ and using the compact support, we obtain
\[
	\begin{split}
		U(z)^{m}
		 & = \frac{1}{b}\frac{1}{\Gamma(1-\alpha)}\int_z^R \left(({1-\alpha}) H_d(\tfrac z \xi) + b \right) \xi^{1 + \frac{1-a}{b}} \left( \int_\xi^R \left(1-\left(\tfrac{\xi}{\rho}\right)^\frac{1}{b}\right)^{-\alpha} \rho^{\frac{a-1}{b}-1} U(\rho)d\rho\right) d \xi                   \\
		 & = \frac{1}{b}\frac{1}{\Gamma(1-\alpha)}\int_z^R \left(\int_z^\rho \left(({1-\alpha}) H_d\left(\tfrac{z}{\xi}\right) + b \right) \left(\tfrac{\xi}{\rho}\right)^{1 + \frac{1-a}{b}} \left(1-\left(\tfrac{\xi}{\rho}\right)^\frac{1}{b}\right)^{-\alpha}d\xi \right)  U(\rho)d\rho.
	\end{split}
\]
Now, recalling that $\mathsf A+d b = (1-\alpha)$ we can write
\begin{equation*}
	U(z)^{m} = \int_z^R K(z,\rho)\, U(\rho) \,d\rho,
\end{equation*}
where the kernel $K(z,\rho)\ge 0$ acts only for $0<z<\rho$ and is given by the exact formula:
\begin{align*}
	K(z,\rho) & := \rho\, Q\left(\tfrac z \rho\right),
	\quad
	Q(\eta) =
	\frac{1}{\Gamma(1-\alpha)} \int_{\eta}^1 \left(\frac{1-\alpha}{b} H_d\left(\eta \sigma^{-1} \right) + 1 \right) \sigma^{1 + \frac{1}{b}-d} \left(1-\sigma^\frac{1}{b}\right)^{-\alpha} d\sigma.
\end{align*}
We arrive at the formulation \eqref{eq:profile equation}.
This is the integral equation that will serve as the basis for the study of the existence, uniqueness, and main properties of $U$.
The equation holds for all $m > m_c$. For profiles of infinite support the same calculation uses $R=\infty$ and Tonelli's theorem.

\paragraph{Simplification to \eqref{eq:Q}.}
Since $0<\eta<1$ we get
\begin{align*}
	\Gamma(1-\alpha) \,  Q'(\eta)
	 & = -\left(\frac{1-\alpha}{b} H_d\left(1 \right) + 1 \right) \eta^{1 + \frac{1}{b}-d} \left(1-\eta^\frac{1}{b}\right)^{-\alpha}
	\\
	 & \quad +\int_{\eta}^1 \frac{1-\alpha}{b} H_d'\left(\eta \sigma^{-1} \right) \sigma^{\frac{1}{b}-d} \left(1-\sigma^\frac{1}{b}\right)^{-\alpha} d\sigma.
\end{align*}
Since $H_d(1) = 0$ and $H_d'(\sigma) = -\sigma^{1-d}$  we arrive at
\begin{equation*}
	K_z(z,\rho)=
	Q'(\eta) =
	-\frac{1}{\Gamma(1-\alpha)} \eta^{1 + \frac{1}{b}-d} \left(1-\eta^\frac{1}{b}\right)^{-\alpha}
	-\eta^{1-d}\frac{1-\alpha}{b\Gamma(1-\alpha)} \, \int_{\eta}^1 \sigma^{ \frac{1}{b}-1} \left(1-\sigma^\frac{1}{b}\right)^{-\alpha} d\sigma.
\end{equation*}
Hence $Q'(\eta) < 0$ for $0<\eta<1$, and we have $K_z(z,\rho)<0$ for $0<z<\rho$. Note that when $d=1$ the integral of $Q'(\eta)$ simplifies since $H'_1(s)=-1$ for $0<s<1$.
Furthermore, we can integrate
$$
	\int_{\eta}^1   \sigma^{ \frac{1}{b}-1} \left(1-\sigma^\frac{1}{b}\right)^{-\alpha} d\sigma=
	b \int_{\eta^{1/b}}^1   t^{1-b} (1-t)^{-\alpha} t^{b-1} dt = \frac{b}{1-\alpha}(1-\eta^{1/b})^{1-\alpha}.
$$
Simplifying we deduce that
\begin{equation}
	\label{eq:Q'}
	-Q'(\eta) =
	\frac{1}{\Gamma(1-\alpha)}\,(1-\eta^{\frac 1 b})^{-\alpha} \eta^{1-d}.
\end{equation}
Furthermore, since $Q(1) = 0$ we have \eqref{eq:Q}.

\begin{remark}
	We can introduce the incomplete beta notation for $x \in (0,1)$
	\begin{equation*}
		B_x(p,q) = \int_0^x \sigma^{p-1} (1-\sigma)^{q-1} d \sigma, \quad I_x(p,q) = \frac {B_x(p,q)} {B(p,q)} .
	\end{equation*}
	Even though for $x < 1$ the value of $q$ is irrelevant to convergence when $p>0$, we use the normalized form only for $p,q > 0$. Using the change of variable $\sigma = \xi^b$ and $s=1-\xi$ we get
	\begin{align*}
		Q(\eta) & =-\int_\eta^1 Q'(\sigma)\,d\sigma
		=\frac b{\Gamma(1-\alpha)}\int_{\eta^{1/b}}^1(1-\xi)^{-\alpha}\xi^{b(2-d)-1}\,d\xi                       \\
		        & =\frac b{\Gamma(1-\alpha)}\int_0^{1-\eta^{1/b}}\sigma^{-\alpha}(1-\sigma)^{b(2-d)-1}\,d\sigma.
	\end{align*}
	We conclude that in dimension $d=1$ we can write
	\begin{equation}
		\label{eq:Q from incomplete gammas}
		Q(\eta) = \frac{b}{\Gamma(1-\alpha)} B_{1-\eta^{1/b}}\Big(1-\alpha, b\Big) = \frac{b \Gamma(b)}{\Gamma(1-\alpha + b)} I_{1-\eta^{1/b}}\Big(1-\alpha, b\Big).
	\end{equation}
\end{remark}

\paragraph{Asymptotic of $Q$ and $\eta = 0, 1$.}
Lastly, we point out the limits
For all dimensions $d \ge 1$, the singularity is dominated by the fractional boundary behavior, yielding the optimal constant
\[
	\lim_{\eta \to 1^-} \frac{Q_{\alpha,m}(\eta)}{(1-\eta)^{1-\alpha}} = \frac{b^\alpha}{\Gamma(2-\alpha)}.
\]
The singular behaviour at $\eta = 0$ is given by the power of the reference profile
\begin{equation}
	\label{eq:Q at 0}
	\lim_{\eta \to 0^+} \frac{Q_{\alpha,m}(\eta)}{\mathcal V_{\alpha,m}(\eta)^m} =
	\begin{dcases}
		\frac{1}{(d-2)\Gamma(1-\alpha)}        & \text{if } d \ge 3 , \\
		\frac{1}{\Gamma(1-\alpha)}             & \text{if } d =2 ,    \\
		\frac{\Gamma(b+1)}{\Gamma(b+1-\alpha)} & \text{if } d=1.
	\end{dcases}
\end{equation}

\subsection{Integro-differential version. Monotonicity and equi-continuity}
Let us move to the study of the properties of the solution to \eqref{eq:profile equation}. If we differentiate with respect to $z$ we get the equivalent formulation
\begin{equation}
	\label{intdiff. equation}
	(U(z)^{m})' = \int_z^{+\infty} K_z(z,\rho) \, U(\rho) d\rho,
\end{equation}
Recalling the above formulas, we get
$$
	K_z(z,\rho)= \rho \, Q'(\eta)\frac1{\rho}= Q'(\eta), \quad \text{ where } \eta := \frac{z}{\rho}.
$$
Since $Q' < 0$ we deduce that
\begin{corollary}
	Any $U \in C((0,\infty))$ solution of \eqref{eq:profile equation} is strictly decreasing in its support.
\end{corollary}

Since $U$ is monotone, we can apply Lieb's trick
\begin{equation*}
	M \ge |\partial B_1| \int_0^z U(\rho) \rho^{d-1} d \rho  \ge U(z) |\partial B_1| \frac{z^d}{d}.
\end{equation*}
Hence, we recover the universal local $L^\infty$ bound depending only on the mass
\begin{equation}
	\label{eq:Lieb upper bound}
	U_{\alpha,m,M} (z) \le z^{-d} \frac{dM}{|\partial B_1|} \quad \forall z > 0.
\end{equation}
From this estimate we also deduce the equi-continuity.
		Write $N=M/|\partial B_1|$ and $F(r)=r^dU(r)\le dN$, and define
		\[
			J(\alpha)=\int_0^1\frac{(1-s)^{-\alpha}-1}{s}\,ds<\infty.
		\]
We point out that
\begin{align*}
	\left| z^{d-1} \frac{dU^m}{dz} (z) \right|
	 & = \int_z^\infty (-Q')(\tfrac{z}{\rho}) \frac{z^{d-1}}{\rho^{d-1}} \rho^{d-1} U(\rho) d\rho \\
	 & =\frac{1}{\Gamma(1-\alpha)}\left[\int_z^\infty U(\rho)\rho^{d-1}\,d\rho
			   +\int_0^1\frac{(1-\eta^{1/b})^{-\alpha}-1}{\eta}F(z/\eta)\,d\eta\right]      \\
	 & \le\frac{N+dbN J(\alpha)}{\Gamma(1-\alpha)}.
\end{align*}
Therefore,
\begin{equation}
	\label{eq:equicontinuity}
	\left| z^{d-1} \frac{dU_{\alpha,m,M}^m}{dz} (z) \right| \le \frac{M[1+dbJ(\alpha)]}{|\partial B_1| \Gamma(1-\alpha)}.
\end{equation}
	Moreover, $J(\alpha)=1/(1-\alpha)+O(1)$ as $\alpha\to1^-$, so $J(\alpha)/\Gamma(1-\alpha)\to1$. Thus this bound is uniform as $\alpha\to1^-$ when $b$ and $M$ remain bounded. On compact intervals in $(0,\infty)$ it gives equi-continuity of $U^m$, and hence of $U$ when $m$ ranges over a compact subset of $(0,\infty)$.

\subsection{Analysis of the singularity of $(U^m)'$ at $z = 0$. Proof of \Cref{eq:limit of flux at z=0}}

The aim of this is to justify \eqref{eq:limit of flux at z=0}.
First, we provide a formal argument through the PDE.
Then, we obtain a direct proof for solutions of \eqref{eq:profile equation} using the properties of $Q$.

\paragraph{Formal estimate through the PDE.}
	The flux constant is consistent with the distributional relation between the Caputo and Riemann--Liouville derivatives for initial datum $M\delta_0$:
	\[
		\partial_t^\alpha u=\dRL u-\frac{M t^{-\alpha}}{\Gamma(1-\alpha)}\delta_0.
	\]
	Formally, the Dirac coefficient in the radial Laplacian is the limiting outward flux through a small sphere. The similarity exponents give
	\[
		|\partial B_1|\varepsilon^{d-1}\partial_r u^m(t,\varepsilon)
		=t^{-\alpha}|\partial B_1|z^{d-1}(U^m)'(z),
		\qquad z=t^{-b}\varepsilon.
	\]
	Matching the Dirac coefficient predicts \eqref{eq:limit of flux at z=0}. The following proof establishes the limit directly from the integral equation, without assuming compact support or interchanging singular space-time limits.

\paragraph{Through the self-similar equation.}
We start with the following result about the behavior of the derivative at the origin.
\begin{lemma}
	\label{lem:limit of flux holds}
	Let $U \in C((0,+\infty))$ be a finite mass solution.
	Then, $z^{d-1} (U^m)' \in C([0,+\infty))$ and \eqref{eq:limit of flux at z=0}.
\end{lemma}
\begin{proof}
	Let us denote, for this proof, $G(\eta) = -\eta^{d-1} Q'(\eta) = \frac{(1-\eta^{\frac 1 b})^{-\alpha}}{\Gamma(1-\alpha)}$ which is non-decreasing in $\eta$.
	For $z > 0$ we write
	\[
		-z^{d-1} (U^m)' =
		\int_z^\infty G(\tfrac z\rho) \rho^{d-1} U(\rho) d\rho
	\]
	This integral is continuous for $z>0$: on any compact interval of positive $z$, its diagonal part is bounded by $C(\rho-z)^{-\alpha}$ times a locally bounded function, while away from the diagonal dominated convergence applies using the finite mass. The diagonal contribution over an interval of length $h$ is uniformly $O(h^{1-\alpha})$.
	In order to pass to the limit, we write
	\begin{equation}
		\label{eq:limit of flux holds decomposition}
		\int_z^\infty G(\tfrac z\rho) \rho^{d-1} U(\rho) d\rho =
		G(0^+) \int_{z}^{+\infty} U(\rho) \rho^{d-1} d \rho + \int_z^\infty (G(\tfrac{z}{\rho}) - G(0^+)) U(\rho) \rho^{d-1}d\rho.
	\end{equation}
	The first term on the right-hand side has a limit because $U$ is non-negative.
	To show that the second term vanishes, we perform the change of variables $\eta = z/\rho$, which yields $d\rho = -z/\eta^2 d\eta$:
	\begin{align*}
		\int_z^\infty (G(\tfrac{z}{\rho}) - G(0^+)) U(\rho) \rho^{d-1}d\rho
		 & = \int_0^1 (G(\eta) - G(0^+)) U(z/\eta) (z/\eta)^{d-1} \frac{z}{\eta^2} d\eta       \\
		 & = \int_0^1 \frac{G(\eta) - G(0^+)}{\eta} \left[ U(z/\eta) (z/\eta)^d \right] d\eta.
	\end{align*}
	Let us define $\Phi(\eta) = \frac{G(\eta) - G(0^+)}{\eta}$ and $F(\rho) = U(\rho)\rho^d$.

	We first analyze the integrability of $\Phi(\eta)$ on $(0,1)$. As $\eta \to 1^-$, we have $\Phi(\eta) \asymp (1-\eta)^{-\alpha}$, which is locally integrable since $\alpha \in (0,1)$. As $\eta \to 0^+$, a Taylor expansion yields $G(\eta) - G(0^+) \asymp \eta^{1/b}$, and thus $\Phi(\eta) \asymp \eta^{1/b - 1}$. Because we are in the range where $b > 0$, this singularity is also locally integrable. Hence, $\Phi \in L^1(0,1)$.

	Next, we consider $F(\rho)$. Because $U$ is non-increasing and $U(\rho)\rho^{d-1} \in L^1(0,\infty)$, we can bound $F(\rho)$ by observing that
	\begin{equation*}
		F(\rho) = U(\rho) \rho^d \le d \int_0^\rho U(s) s^{d-1} ds.
	\end{equation*}
	We conclude that $F(\rho)$ is globally bounded, $F(\rho) \le C(d) M$, and that $\lim_{\rho \to 0^+} F(\rho) = 0$.

	Therefore, for any fixed $\eta \in (0,1)$, the integrand satisfies $\Phi(\eta) F(z/\eta) \to 0$ pointwise as $z \to 0^+$. Furthermore, the integrand is uniformly bounded by $dM\Phi(\eta)/|\partial B_1|\in L^1(0,1)$. By the dominated convergence theorem, the second term in the right-hand side of \eqref{eq:limit of flux holds decomposition} vanishes in the limit $z \to 0^+$. The first term tends to $M/(|\partial B_1|\Gamma(1-\alpha))$, proving \eqref{eq:limit of flux at z=0}.
\end{proof}
%
%
%
%
%
%
%
%

%

\subsection{Analysis of the singularity of $U$ at $z = 0$. Proof of \Cref{eq:profile at z=0}}
\label{sec:U at z=0}
We now use \eqref{eq:Q at 0} to deduce \eqref{eq:profile at z=0}, similarly to \Cref{lem:limit of flux holds}.
\paragraph{Case $d \ge 3$.} We have that
\begin{align*}
	\frac{U_{\alpha,m,M}(z)^m}{\mathcal V_{\alpha,m}(z)^m} & = \frac{ U (z)^m }{z^{2-d}} = \int_{z}^\infty (\tfrac{z}{\rho} )^{d-2}Q(\tfrac{z}{\rho})  U(\rho) \rho^{d-1} d \rho
	\\
	                                                       & \to \frac{M}{|\partial B_1| (d-2)\Gamma(1-\alpha)} .
\end{align*}
Indeed, $\eta^{d-2}Q(\eta)$ extends continuously to $[0,1]$ by \eqref{eq:Q at 0} and $Q(1)=0$, so dominated convergence applies with majorant $C U(\rho)\rho^{d-1}$. Equivalently, integrating the flux asymptotic \eqref{eq:limit of flux at z=0} gives
	\[
		U(z)^m\sim\frac{M}{|\partial B_1|(d-2)\Gamma(1-\alpha)}z^{-(d-2)},
		\qquad U(z)\asymp z^{-(d-2)/m}.
	\]
	Thus the radial mass density behaves as $z^{d-1-(d-2)/m}$, which is integrable at zero precisely when $m>(d-2)/d=m_c$.
As $\alpha \to 1^-$ this limit vanishes because $U_{1,m,M}(z)$ is bounded at $z = 0$.

\paragraph{Case $d =2$.} We have that
\begin{align*}
	\frac{U_{\alpha,m,M}(z)^m}{\mathcal V_{\alpha,m}(z)^m} & = \frac{ U (z)^m }{-\log z}
	= \int_{z}^\infty \frac{1}{-\log \frac z \rho - \log \rho} Q(\tfrac{z}{\rho})  U(\rho) \rho^{d-1} d \rho
	\\
	                                                       & \to \frac{M}{|\partial B_1| \Gamma(1-\alpha)} .
\end{align*}
As $\alpha \to 1^-$ this limit vanishes because $U_{1,m,M}(z)$ is bounded at $z = 0$.

\paragraph{Case $d = 1$.}
Let us put $U=U_{\alpha,m,M}$. The first moment is finite near the origin because of \eqref{eq:profile mass}. For $\rho\ge2$ we have $Q(1/\rho)\ge Q(1/2)>0$,
and therefore \eqref{eq:profile equation} at $z=1$ gives
\[
	Q(1/2)\int_2^\infty \rho\,U(\rho)\,d\rho \le U(1)^m,
\]
which implies $\int_0^\infty \rho\,U(\rho)\,d\rho<\infty$. Next we observe that, the comparison function in
\eqref{eq:profile at z=0} equals one. By \eqref{eq:Q at 0} and the monotone
convergence theorem in \eqref{eq:profile equation}, we obtain for $z\rightarrow 0^+$
\[
	U(z)^m=\int_z^\infty \rho\,Q(z/\rho)\,U(\rho)\,d\rho\ \rightarrow\ Q(0^+)\int_0^\infty \rho\,U(\rho)\,d\rho=\frac{\Gamma(b+1)}{\Gamma(b+1-\alpha)}\,\int_0^\infty \rho\,U(\rho)\,d\rho,
\]
	Let $M_{\mathrm{ref}}=2\int_0^\infty \mathcal U_{\alpha,m}(\rho)\,d\rho$ be the mass of the reference profile in dimension one. By \eqref{eq:rescaling profile to correct mass}, we have $A^{(m+1)/2}=M/M_{\mathrm{ref}}$, and hence
	\[
		U_{\alpha,m,M}(0^+)^m
		=A^m\mathcal U_{\alpha,m}(0^+)^m
		=C(\alpha,m)M^{\frac{2m}{m+1}},
		\qquad
		C(\alpha,m)=\frac{\mathcal U_{\alpha,m}(0^+)^m}{M_{\mathrm{ref}}^{\frac{2m}{m+1}}}.
	\]
	This proves the $d=1$ case of \eqref{eq:profile at z=0}.
\qed

\subsection{Kernel estimates}

We focus on the representation using \eqref{eq:Q}.
We begin this section with some notation
For $f, g: (a,b) \to [0,+\infty)$, introduce the notation
\[
	f \asympint{a}{b} g,
\]
if there exists $c(a,b)$, $C(a,b) > 0$ such that
\begin{equation*}
	c(a,b) g(z) \le f(z) \le C(a,b) g(z) \quad \forall z \in (a,b).
\end{equation*}
We will set $\asymp$ to mean $\asympint{0}{+\infty}$.

\bigskip

For $\eta \sim 1^-$ (specifically for $\eta \in [1/2, 1)$), we have
\begin{equation*}
	\sigma^{1 + \frac1 b - d} \asympint{\frac 1 2}{1} 1 \quad \text{and} \quad 1-\sigma^{\frac 1 b} \asympint{\frac 1 2}{1} 1-\sigma.
\end{equation*}
Furthermore, because $\sigma \in (\eta, 1)$, as $\eta \to 1^-$ we also have $\sigma \to 1^-$, which implies $\eta\sigma^{-1} \to 1$.
Hence, we can estimate $Q(\eta)$ by focusing on the remaining singular term. Using the change of variable $(1-\eta)\xi = 1-\sigma$ (which gives $d\sigma = -(1-\eta)d\xi$), we recover:
\begin{equation*}
	\begin{aligned}
		Q(\eta) & \asympint{\frac{1}{2}}{1} \int_\eta^1 (1-\sigma)^{-\alpha} \, d\sigma
		\asympint{\frac{1}{2}}{1} \int_0^1 \big((1-\eta)\xi\big)^{-\alpha} (1-\eta) \, d\xi
		\asympint{\frac{1}{2}}{1} (1-\eta)^{1-\alpha} \int_0^1 \xi^{-\alpha} \, d\xi.
	\end{aligned}
\end{equation*}
We deduce that
\begin{equation}
	\label{eq:Q away from 0}
	Q(\eta) \asympint{\frac{1}{2}}{1} (1 - \eta)^{1-\alpha}
\end{equation}
We now estimate the behavior of $Q(0^+)$. When $0 < \eta < \frac 1 2$ then we write
\begin{align*}
	Q(\eta) & = \frac{1}{\Gamma(1-\alpha)}\left( \int_\eta^{\frac 1 2} (1-\sigma^{\frac 1 b})^{-\alpha} \sigma^{1-d} d\sigma +  \int_{\frac 1 2}^1 (1-\sigma^{\frac 1 b})^{-\alpha} \sigma^{1-d} d\sigma \right)
	\asympint{0}{\frac 1 2} \int_\eta^1 \sigma^{1-d} d \sigma + 1.
\end{align*}
Hence, we deduce that
\begin{equation}
	\label{eq:Q close to 0}
	Q(\eta) \asympint{0}{\frac 1 2}
	\begin{dcases}
		1          & \text{if } d = 1,   \\
		-\log \eta & \text{if } d = 2,   \\
		\eta^{2-d} & \text{if } d \ge 3.
	\end{dcases}
	\quad
	.
\end{equation}
Hence, if $d \ge 2$ we have $K(0^+, \rho) = +\infty$ for each $\rho > 0$. This means that any non-trivial solution satisfies $U(0^+) = +\infty$.

\subsection{Limits of the kernel as $\alpha \to 1^-$}
	The measures $s^{-\alpha}\,ds/\Gamma(1-\alpha)$ on $[0,1]$ converge weakly to $\delta_0$ as $\alpha\to1^-$. We apply this concentration property directly to the kernel, keeping $m$ and $d$ fixed. The expected limit is
\begin{equation}
	\label{eq:K limit as alpha to 1}
	K_{\alpha,m}(z,\rho) \to K_{1,m} (\rho) \coloneqq b_{1,m} \rho \text{ as } \alpha \to 1^- \text{ for each } 0 < z < \rho.
\end{equation}
We prove this rigorously.
\begin{lemma}
	\label{eq:Q limit alpha -> 1}
	Assume $Q(\eta)$ vanishes for $\eta \ge 1$, and for $\eta \in (0,1)$ its derivative is given according to \eqref{eq:Q'}.
	Then, we have that
	\[
		\lim_{\alpha \to 1^-} Q_{\alpha,m}(\eta) = Q_{1,m} \coloneqq b_{1,m}  \quad \text{ for all } \eta \in (0,1).
	\]
\end{lemma}

\begin{proof}
	For this proof we will denote $b_\alpha = b_{\alpha,m}$. Using the change of variables $s = 1 - \sigma^{\frac{1}{b_\alpha}}$, which implies $\sigma = (1-s)^{b_\alpha}$ and $d\sigma = -b_\alpha(1-s)^{b_\alpha-1} ds$, the integration limits change from $[\eta, 1]$ to $[1 - \eta^{1/b_\alpha}, 0]$. We then have
	\begin{align*}
		Q_{\alpha,m}(\eta) & = \int_0^{1-\eta^{1/b_\alpha}} \frac{s^{-\alpha}}{\Gamma(1-\alpha)} (1-s)^{b_\alpha(1-d)} b_\alpha(1-s)^{b_\alpha-1} ds \\
		                   & = \int_0^{1-\eta^{1/b_\alpha}} \frac{s^{-\alpha}}{\Gamma(1-\alpha)} \Psi_\alpha(s) ds
		\quad \text{ where } \Psi_\alpha(s) = b_\alpha (1-s)^{b_\alpha(2-d)-1}.
	\end{align*}
	We decompose the integral into two parts:
	\begin{equation*}
		\left| Q_{\alpha,m} (\eta) - \Psi_\alpha(0) \int_0^{1-\eta^{1/b_\alpha}} \frac{s^{-\alpha}}{\Gamma(1-\alpha)} ds \right|
		\le
		\int_0^{1-\eta^{1/b_\alpha}}
		\frac{s^{-\alpha}}{\Gamma(1-\alpha)} \left| \Psi_\alpha(s) - \Psi_\alpha(0) \right|
		ds.
	\end{equation*}
	Observing that $\Psi_\alpha(0) = b_{\alpha,m}$, the first integral evaluates exactly to:
	\begin{align*}
		\Psi_\alpha(0) \int_0^{1-\eta^{1/b_\alpha}} \frac{s^{-\alpha}}{\Gamma(1-\alpha)} ds
		 & =
		b_{\alpha,m} \frac{(1-\eta^{1/b_\alpha})^{1-\alpha}}{(1-\alpha)\Gamma(1-\alpha)}
		=
		b_{\alpha,m} \frac{(1-\eta^{1/b_\alpha})^{1-\alpha}}{\Gamma(2-\alpha)}.
	\end{align*}
	As $\alpha \to 1^-$, we have $b_{\alpha,m} \to b_{1,m}$, the exponent $(1-\alpha) \to 0^+$, and $\Gamma(2-\alpha) \to \Gamma(1) = 1$. Since $\eta \in (0,1)$, the base $(1-\eta^{1/b_\alpha})$ is strictly positive, meaning $(1-\eta^{1/b_\alpha})^{1-\alpha} \to 1$. Thus, this leading term converges exactly to $b_{1,m}$.

	For the second term, we point out that $\Psi_\alpha$ is smooth near $s=0$. Because $\eta \in (0,1)$, the upper integration limit is uniformly bounded away from $1$: there exists an $\varepsilon > 0$ such that $1-\eta^{1/b_\alpha} \le 1 - \varepsilon$ for all $\alpha$ close to $1$. On the interval $[0, 1-\varepsilon]$, $\Psi_\alpha$ is Lipschitz continuous uniformly in $\alpha$. Thus, there exists a constant $C > 0$ independent of $\alpha$ such that $|\Psi_\alpha(s) - \Psi_\alpha(0)| \le C s$.
	Therefore, we can bound the remainder by:
	\begin{align*}
		\int_0^{1-\varepsilon} \frac{s^{-\alpha}}{\Gamma(1-\alpha)} \left| \Psi_\alpha(s) - \Psi_\alpha(0) \right| ds
		 & \le
		\frac{C}{\Gamma(1-\alpha)} \int_0^{1-\varepsilon} s^{1-\alpha} ds
		\le \frac{C}{\Gamma(1-\alpha)(2-\alpha)}.
	\end{align*}
	As $\alpha \to 1^-$, we have $\Gamma(1-\alpha) \to \infty$. This causes the upper bound to vanish. Combining both parts, we conclude that $\lim_{\alpha \to 1^-} Q_{\alpha,m}(\eta) = b_{1,m}$.
\end{proof}

\begin{remark}
	Notice that the convergence is not uniform in $[0,1]$ since $Q_{\alpha,m}(1) = 0$ for all $\alpha < 1$.
		After extending $Q_{\alpha,m}$ by zero for $\eta\ge1$, a similar argument shows that $-Q'_{\alpha,m} \to b_{1,m} \delta_1$ in $\mathcal D'(0,\infty)$.
\end{remark}

\section{Linear diffusion $m = 1$}
\label{sec:m=1}

\subsection{Proof of \Cref{thm:existence} for $m = 1$}
\label{sec:existence proof m=1}
\paragraph{Solution by Fourier transform.}
The existence of a profile is well-known, see \cite{MR1061448, MR2047909,MR3631303}. In this case, we can take the Fourier transform of \eqref{eq:main} to deduce that
\begin{equation*}
	\partial_t^\alpha \widehat u(t, \xi) = -|\xi|^2 \widehat u(t,\xi).
\end{equation*}
Hence, the solution with general initial datum $u(0,x) = u_0(x)$ can be recovered by separation of variables $\widehat u(t,\xi) = \widehat Z(t,\xi) \widehat u_0(\xi)$
where for each $\xi$ fixed $\widehat Z(t,\xi)$ is the solution to
\[
	\begin{dcases}
		\partial_t^\alpha \widehat Z(t,\xi) = -|\xi|^2 \widehat Z(t,\xi), & \text{if } t > 0 \\
		\widehat Z(0,\xi)=1.
	\end{dcases}
\]
We have precisely the following
\begin{equation*}
	\widehat Z(t, \xi) = E_\alpha(-|\xi|^2t^\alpha), \quad E_\alpha(s) = \sum_{k=0}^\infty \frac{s^k}{\Gamma(\alpha k +1)}.
\end{equation*}
where $E_\alpha$ is the Mittag-Leffler function.
And this yields the following
\[
	u(t,x) = \int_{\Rd} Z(t,x-y) u_0(y) dy.
\]
The function $Z$ is the so-called \emph{fundamental solution}.

\paragraph{Reference profile.}
$Z$ is a classical solution for $x \ne 0$ and $t > 0$. We take $\mathcal U_{\alpha,1}(z) = Z(1,z e_1)$ and this is a solution to \eqref{eq:profile equation}.
This already has mass $1$. And in \eqref{eq:rescaling profile to correct mass} we get $A = M$ and
\begin{equation*}
	U_{\alpha,1,M} (z) = M \mathcal U_{\alpha,1}(z).
\end{equation*}
We can apply \Cref{lem:sss is weak solution} and \Cref{lem:limit of flux holds} to conclude the result.
\qed

\subsection{Uniqueness. Proof of \Cref{thm:m=1 uniqueness}}
\label{sec:linear uniqueness}
Let $u_1, u_2$ be two solutions to \eqref{eq:very weak solution} in the stated finite-mass class with the same initial mass.
The difference $w = u_1 - u_2$ satisfies
\begin{equation*}
	\int_{0}^T  \int_{\Rd} w(t,x) \dRL \varphi(T-t,x)dxdt
	= \int_0^T \int_{\Rd} w(t,x) \Delta\varphi(T-t,x)  dx dt
\end{equation*}
To prove uniqueness, we use as test function in the weak formulation the solution to the following adjoint problem
\begin{equation*}
	\dRL \varphi_\varepsilon(t,x) - \Delta \varphi_\varepsilon(t,x) = f_\varepsilon(t,x) , \quad \varphi_{\varepsilon}(0,x) = 0,
\end{equation*}
for a function $f_\varepsilon$ that we construct in the following. Notice that since starting from $0$ initial datum this problem coincides with the Caputo one.

We construct $f_\varepsilon$ as follows. Fix a non-negative mollifier $\zeta \in C^\infty_c (\mathbb R^{d+1})$ with $\int_{\mathbb R^{d+1}} \zeta(t,x) dt dx = 1$. Define $\zeta_\varepsilon(t,x) = \varepsilon^{-(d+1)} \zeta( t / \varepsilon , x / \varepsilon)$.
For $0 < R_1 < R_2$ and $\varepsilon > 0$ we construct
\begin{equation*}
	f_\varepsilon(t,x) = t^2 \int_{0}^T \int_{B_{R_2} \setminus B_{R_1}} \zeta_\varepsilon(t-s,y-x) g(T-s,y) dy\,ds, \quad g(t,x) =  \operatorname{sign} w(t,x).
\end{equation*}
Since $f_\varepsilon(0)=\partial_t f_\varepsilon(0)=0$, we can construct this solution using the mass-one fundamental solution $Z$:
	\[
		\begin{aligned}
			\varphi_\varepsilon(t)           & =\int_0^t Z(t-s)*_x I_\alpha\partial_t f_\varepsilon(s)\,ds,    \\
			\partial_t\varphi_\varepsilon(t) & =\int_0^t Z(t-s)*_x I_\alpha\partial_{tt} f_\varepsilon(s)\,ds.
		\end{aligned}
	\]
	The Fourier--Laplace transform of the first expression is $\widetilde{\widehat f_\varepsilon}(p,\xi)/(p^\alpha+|\xi|^2)$, verifying the adjoint equation. Since $\|Z(t)\|_1=1$, these formulas bound $\varphi_\varepsilon$, $\partial_t\varphi_\varepsilon$, $\nabla\varphi_\varepsilon$ and $\Delta\varphi_\varepsilon$ uniformly on $[0,T]\times\Rd$.
Due to the slow growth in $t$ of $f$ at $t = 0$ we have that $\varphi_\varepsilon \in X \coloneqq C^1((0,T) \times \mathbb R^d) \cap C((0,T); C^2(\mathbb R^d))$. Since test functions in \eqref{eq:very weak solution} are required to be compactly supported, we can approximate $\varphi_\varepsilon$ in $X$ by a sequence of compactly supported functions. First multiply by $\chi(x/R)$, where $\chi\in C_c^\infty$ equals one near zero: the additional Laplacian terms are bounded by $C/R+C/R^2$. Then extend by zero for $t<0$, smooth in time and subtract the smoothed value at $t=0$; the time derivative and spatial derivatives up to order two converge uniformly. Since $\dRL\varphi_\varepsilon=I_{1-\alpha}\partial_t\varphi_\varepsilon$, this also controls the fractional derivatives.
Hence $\partial_t^\alpha \varphi_\varepsilon$ and $\Delta \varphi_\varepsilon$ are approximated uniformly in compacts with uniform bounds in $L^\infty((0,T) \times \mathbb R^d)$. Since $w \in L^1((0,T) \times \mathbb R^d)$ the weak formulation holds for $\varphi_\varepsilon$ using the dominated convergence theorem. Due to the weak formulation, we have that
\begin{equation*}
	\int_0^T \int_{\mathbb R^d} w(t,x) f_\varepsilon(T-t,x) dx dt = 0.
\end{equation*}
Then letting $\varepsilon \to 0$
we deduce that
\begin{equation*}
	\int_0^T \int_{B_{R_2} \setminus B_{R_1}} (T-t)^2 |w(t,x)|\,dx\,dt = 0.
\end{equation*}
Since this works for $R_1, R_2$ general, $w = 0$ and the proof is complete.

	To apply this result to a continuous finite-mass integral profile $U$, the flux estimate in the proof of \Cref{lem:limit of flux holds} gives $|U'(\rho)|\le C\rho^{1-d}$. Since $a=db$ and $2b=\alpha$, differentiation of $F$ is justified by
		\[
			(1-\sigma)^{-\alpha}\sigma^{-a-b}|U'(z\sigma^{-b})|
			\le C_z\sigma^{-\alpha}(1-\sigma)^{-\alpha}\in L^1(0,1),
		\]
		locally uniformly for $z>0$. Thus $F\in C^1$; Tonelli's theorem recovers \eqref{eq:equation for dUm/dz} from \eqref{eq:Q'}, and differentiation gives $U\in C^2$ and \eqref{eq:profile equation second order}. Monotonicity and finite mass give time absolute continuity at fixed $x\ne0$, as in the proof of \Cref{lem:sss is weak solution}. The spatial flux is continuous at any finite support boundary, since $U'$ is continuous and zero outside. Together with \Cref{lem:limit of flux holds}, these facts verify the hypotheses of the weak-solution lemma.
Given two profiles $U_1, U_2$ of mass $M$, we can use \Cref{lem:sss is weak solution} to deduce that the associated solutions $u_1, u_2$ are weak solutions. By mass conservation they belong to $L^\infty(0,T;L^1(\Rd))$. Therefore, we deduce $u_1 = u_2$ and therefore $U_1 = U_2$.
\qed

\section{Slow diffusion $m>1$}
In this setting, we will show that the profiles are compactly supported and therefore all the computations for the derivation of \eqref{eq:profile equation} are justified.
\subsection{Uniqueness of solutions to the profile equation}
Let
\begin{equation*}
	\fixedpoint U (r) = \left( \int_{r}^\infty K(r, \rho) U(\rho) d\rho  \right)^{1/m}.
\end{equation*}
We have the following fundamental result.
\begin{theorem}[Forward uniqueness]
	\label{thm:profile forward uniqueness}
	Assume $m > 1$.
	Assume that $U_1 \in C((0,\infty))$ with $\text{supp}(U_1) = [0,R]$ for some $R<\infty$ and $U_2 \in C((0,\infty))$ are non-negative solutions of $U_i = \fixedpoint U_i$ in $[z_0, +\infty)$, where $z_0>0$. If $U_1(z_0) = U_2(z_0)$ then $U_1 = U_2$ in $[z_0,+\infty)$.
\end{theorem}

We start with a lemma that strongly uses the homogeneity of the operator.
\begin{lemma}
	\label{lem:comparison by homogeneity}
	Assume $m > 1$.
	Assume $U_1, U_2 \in C([z_0,\infty))$ are non-negative, such that $U_1\le\fixedpoint U_1$ and $U_2\ge\fixedpoint U_2$ in $[z_0,+\infty)$, $U_1=0$ on $[R,\infty)$ for some finite $R>z_0$, and the set
	\[
		\Lambda = \{ \lambda > 0 : \lambda U_1 \le U_2 \text{ in } [z_0,R] \}
	\]
	is non-empty, then $U_1 \le U_2$ on $[z_0,\infty)$.
\end{lemma}
\begin{proof}
	If $U_1\equiv0$, the result is immediate. Otherwise, consider $\lambda_* = \sup \Lambda > 0$. This supremum is finite, and $\lambda_*U_1\le U_2$ on $[z_0,\infty)$ by the support assumption. Since $K \ge 0$ we have that
	\begin{equation*}
		U_2 \ge \fixedpoint U_2 \ge \fixedpoint (\lambda_* U_1) = \lambda_*^{\frac 1 {m}} \fixedpoint U_1 \ge \lambda_*^{\frac 1 {m}} U_1.
	\end{equation*}
	Therefore, $\lambda_*^{\frac 1 {m}} \in \Lambda$, we deduce $\lambda_* \ge \lambda_*^{\frac 1 {m}}$, and hence, since $m > 1$, $\lambda_* \in \{0\} \cup [1,+\infty)$. We conclude that $\lambda_* \ge 1$ and so $U_1 \le U_2$.
\end{proof}

We now proceed to the proof of the main result.
\begin{proof}[Proof of \Cref{thm:profile forward uniqueness}]
	If the common value at $z_0$ is zero, positivity of $K$ implies that both profiles vanish on $[z_0,\infty)$. Otherwise, first, let us prove that $U_1$ and $U_2$ have the same support to the right of $z_0$. Since $K$ is positive then the supports are of the form ${[z_0},R_i]$ with $R_1 \in (z_0,\infty)$ and $R_2 \in [z_0,+\infty]$, and $U_i > 0$ in $(z_0,R_i)$.

	Assume, towards a contradiction, that $R_1 < R_2$. Then, we have $U_2/U_1 \ge c > 0$ in $[z_0,R_1)$, since $U_2$ has a positive minimum and $U_1$ is bounded on $[z_0,R_1]$.
	Since outside $\text{supp}(U_1)$ the claim $\lambda U_1 \le U_2$ is trivial, then $c \in \Lambda$. Therefore, by the previous lemma $U_1 \le U_2$. Moreover, $U_1<U_2$ on $(R_1,R_2)$. Then, since $K(z_0, \rho) > 0$ for $\rho \in (z_0,+\infty)$,
	$$
		U_1(z_0) = \fixedpoint U_1 (z_0) < \fixedpoint U_2 (z_0) = U_2 (z_0).
	$$
	But $U_1(z_0) = U_2(z_0)$. This is a contradiction.

	If $R_2 < R_1$ then $R_2$ is finite, and we can switch the roles in the previous argument. Hence, we deduce that $R_1 = R_2$.

	Now let us consider $ U_2^{(\varepsilon)}$ a re-scaling whose support ends at $(1+\varepsilon)R$ using \eqref{eq:scaling solutions}.
	Explicitly,
	\[
		U_2^{(\varepsilon)}(z)=(1+\varepsilon)^{2/(m-1)}U_2\bigl(z/(1+\varepsilon)\bigr),
	\]
	which solves the equation on $[(1+\varepsilon)z_0,\infty)$. Take $\varepsilon>0$ sufficiently small that $(1+\varepsilon)z_0<R$.
	Since $U_2^{(\varepsilon)}$ is non-increasing, and it is positive at $z = R$, then it is strictly positive in $[(1+\varepsilon)z_0,R]$. Since $U_1$ is bounded in $[(1+\varepsilon)z_0,R]$, then the set $\Lambda$ in \Cref{lem:comparison by homogeneity} is non-empty. And $U_1 \le U_2^{(\varepsilon)}$ on $[(1+\varepsilon)z_0,\infty)$. Letting $\varepsilon \to 0$ we deduce $U_1 \le U_2$ for $z>z_0$, and continuity includes $z_0$. By switching indices and repeating the argument $U_1 = U_2$.
\end{proof}

\begin{corollary}
	\label{thm:comparison principle}
	Assume $m > 1$.
	Let $U_1, U_2 \in C((0,\infty))$ be non-negative solutions with a bounded support of $U_i = \fixedpoint U_i$ in $(0,+\infty)$. If $U_1(z_0) < U_2(z_0)$ for some $z_0 > 0$ then for each $z\in (0, \infty)$ either $U_1(z) < U_2(z)$ or $U_1(z) = U_2(z) = 0$.
\end{corollary}

\begin{proof}
	It suffices to exclude an intersection at a positive value. Assume, towards a contradiction that $U_1(z_1) = U_2(z_1) > 0$ for some $z_1>0$. By \Cref{thm:profile forward uniqueness}, $U_1 = U_2$ in $[z_1,+\infty)$.
	If $z_0 \ge z_1$ we have a contradiction and the proof is complete.
	If $z_0 < z_1$, since $U_1$ is non-increasing, it is bounded below by a positive constant in $[z_0,z_1]$.
	And $U_2$ is bounded above and below in $[z_0,z_1]$.
	Hence $U_1\ge cU_2$ on $[z_0,+\infty)$ for some $c>0$.
	We can therefore apply \Cref{lem:comparison by homogeneity} to deduce $U_2 \le U_1$ in $[z_0, +\infty)$.
	This is a contradiction. Thus equality occurs only where both profiles vanish, and continuity gives the assertion.
\end{proof}

\subsection{A reference shape for the solution of support $[0,1]$}
We define the reference shape functions
\begin{equation}
	\label{eq:slow reference shape}
	\mathcal V_{\alpha,m} (z) =
	\begin{dcases}
		z^{-(d-2)/m}                                            & \text{if } 0<z<1/2,\ d\ge3, \\
		|\log z|^{1/m}                                          & \text{if } 0<z<1/2,\ d=2,   \\
		1                                                       & \text{if } 0<z<1/2,\ d=1,   \\
		C(d,\alpha,m) (1-z)_+^{\frac{2-\alpha}{m-1}} & \text{if } z \ge 1/2.
	\end{dcases}
\end{equation}
where $C(d,\alpha,m)$ is chosen so that the profile is continuous at $z = 1/2$.
Explicitly, $C(d,\alpha,m)=2^{(2-\alpha)/(m-1)}\mathcal V_{\alpha,m}(1/2^-)$; in particular, $C(d,\alpha,m)=2^{(d-2)/m+(2-\alpha)/(m-1)}$ for $d\ge3$.
Notice that $\mathcal V_{\alpha,m}$ is non-increasing. We will now show that for suitable $\overline C$ we have that $\overline{\mathcal U}(z) = \overline C \mathcal V_{\alpha,m}(z)$ is a supersolution.

\begin{lemma}
	Let $\alpha \in (0,1)$, $m > 1$, and $d \ge 1$. We have $\fixedpoint \mathcal V_{\alpha,m} \asymp \mathcal V_{\alpha,m}$.
\end{lemma}

\begin{proof}
	For the duration of this proof let $p = \frac{2-\alpha}{m-1}$.
	We separate the proof into different ranges of $z$. We will show the proof of the upper bounds, and the lower bounds are proved similarly.

	\subparagraph{Range $z \in (\frac 1 4,1)$ and $d \ge 1$.}
	On $[1/4,1)$ the ratio $\mathcal V_{\alpha,m}(\rho)/(1-\rho)^p$ is bounded above and below by positive constants. Also, \eqref{eq:Q away from 0} extends to $1/4\le\eta<1$ since $Q$ is positive and continuous on $[1/4,1/2]$. With $K(z,\rho)=\rho Q(z/\rho)$ and $1-\rho=(1-z)\sigma$, we obtain
	\begin{align*}
		\int_z^1 K(z,\rho) \mathcal V_{\alpha,m} (\rho) d \rho
		 & \asympint{\frac14}{1}   \int_z^1 \rho Q(z/\rho) (1-\rho)_+^{\frac{2-\alpha}{m-1}} d \rho
		\\
		 & \le  C(\alpha,m,d)   \int_z^1 \rho (1-z/\rho)^{1-\alpha} (1-\rho)_+^{\frac{2-\alpha}{m-1}} d \rho
		\\
		 & \le  C(\alpha,m,d) (1-z)^{1-\alpha + \frac{2-\alpha}{m-1} + 1} \int_0^1 (1 - \sigma)^{1-\alpha} \sigma^{\frac{2-\alpha}{m-1}} d \sigma
		\\
		 & \le {C(\alpha,m,d) } \mathcal V_{\alpha,m}(z)^m.
	\end{align*}
	And the lower bound follows similarly bounding $\rho \ge z \ge \frac 1 4$.

	\subparagraph{Range $z \in (0,1/4)$ when $d = 1$.}
	We have that
	\begin{align*}
		\int_z^1 K(z,\rho) \mathcal V (\rho) d \rho
		 & \le C \left(  \int_z^{2z}  Q(z/ \rho) \rho d \rho + \int_{2z}^{1/2}   Q(z/\rho) \rho d \rho + \int_{1/2}^{1}  Q(z/\rho) (1-\rho)^{\frac{2-\alpha}{m-1}} d \rho
		\right)
		\\
		 & \le C\left( \int_z^{2z}  (1 - z/\rho)^{1-\alpha}  \rho d \rho + \int_{2z}^{1/2}\rho d \rho+ \int_{1/2}^{1} (1-\rho)^{\frac{2-\alpha}{m-1}} d \rho
		\right)
		\\
		 & \le C \left(
		z^2
		+  1 \right)
		\le C \mathcal V(z)^m.
	\end{align*}
	The lower bound follows equivalently.

	\subparagraph{Range $z \in (0,1/4)$ when $d \ge 3$.}
	Put $\gamma=(d-2)/m$, so that $\mathcal V_{\alpha,m}(z)=z^{-\gamma}$ and $\mathcal V_{\alpha,m}(z)^m=z^{2-d}$ for $0<z<1/2$. We split the computation into three parts
	\begin{align*}
		\int_z^1 K(z,\rho) \mathcal V_{\alpha,m} (\rho) d \rho
		 & \le C\left( \int_z^{2z}  K(z, \rho) \rho^{-\gamma} d \rho
		+ \int_{2z}^{1/2}   K(z, \rho) \rho^{-\gamma} d \rho
		+ \int_{1/2}^{1}  K(z, \rho) (1-\rho)^{\frac{2-\alpha}{m-1}} d \rho \right).
	\end{align*}
	Now we deal with each term separately
	\begin{align*}
		\int_z^{2z}  K(z, \rho) \rho^{-\gamma} d \rho
		 & \le C \int_z^{2z}  Q(z/ \rho) \rho^{1-\gamma} d \rho
		\le C\int_z^{2z}  (1 - z/\rho)^{1-\alpha}  \rho^{1-\gamma} d \rho
		\le C z^{2-\gamma}.
	\end{align*}
	We also have that
	\begin{align*}
		\int_{2z}^{1/2}   K(z, \rho) \rho^{-\gamma} d \rho
		 & \le C \int_{2z}^{1/2}   Q(z/\rho) \rho^{1-\gamma} d \rho
		\le C \int_{2z}^{1/2}(z/\rho)^{2-d}\rho^{1-\gamma} d \rho
		\le C z^{2-d} .
	\end{align*}
	Lastly
	\begin{align*}
		\int_{1/2}^{1}  K(z, \rho) (1-\rho)^{\frac{2-\alpha}{m-1}} d \rho
		 & \le C \int_{1/2}^{1}  Q(z/\rho) (1-\rho)^{\frac{2-\alpha}{m-1}} d \rho
		\le C \int_{1/2}^{1} (z/\rho)^{2-d} (1-\rho)^{\frac{2-\alpha}{m-1}} d \rho
		\le Cz^{2-d}.
	\end{align*}
	Because $m > 1$, we have $\gamma < d-2$, which implies that the term $z^{2-d}$ strictly dominates the term $z^{2-\gamma}$ as $z \to 0^+$. Given that $\gamma m = d-2$, this yields the correct target scaling. We conclude that $(\fixedpoint\mathcal V_{\alpha,m}(z))^m \le C z^{2-d} = C \mathcal V_{\alpha,m}(z)^m$. For the lower bound, \eqref{eq:Q close to 0} gives
			\[
				(\fixedpoint\mathcal V_{\alpha,m}(z))^m
				\ge\int_{1/2}^{3/4}\rho Q(z/\rho)\mathcal V_{\alpha,m}(\rho)\,d\rho
				\ge c z^{2-d},\qquad 0<z<1/4.
			\]
			The last constant is positive because $\mathcal V_{\alpha,m}$ is positive on $[1/2,3/4]$.

	\subparagraph{Range $z \in (0,1/4)$ when $d = 2$.}
	We proceed in a similar way, computing each term
	\begin{align*}
		\int_z^1 & K(z,\rho) \mathcal V_{\alpha,m} (\rho) d \rho
		\\
		         & \le C \left( \int_z^{2z}  K(z, \rho) (-\log \rho)^{\frac 1 m} d \rho + \int_{2z}^{1/2}   K(z, \rho) (-\log \rho) ^{\frac 1 m}d \rho + \int_{1/2}^{1}  K(z, \rho) (1-\rho)^{\frac{2-\alpha}{m-1}} d \rho \right).
	\end{align*}
	We now estimate the terms separately. First, we deal with the influence of $\rho$ closest to $z$
	\begin{align*}
		\int_z^{2z}  K(z, \rho) (-\log \rho)^{\frac 1 m} d \rho
		 & \le C \int_z^{2z}  Q(z/ \rho) \rho(-\log \rho)^{\frac 1 m} d \rho
		\\
		 & \le C  \int_z^{2z}  (1 - z/\rho)^{1-\alpha}  \rho(-\log \rho)^{\frac 1 m} d \rho
		\\
		 & \le C z^2 \int_1^{2}  (1 - \sigma^{-1})^{1-\alpha}  \sigma (-\log z - \log \sigma)^{\frac 1 m} d \sigma
		\\
		 & \le  C z^2 (-\log z)^{\frac 1 m} \int_1^{2}  (1 - \sigma^{-1})^{1-\alpha}  \sigma  d \sigma.
	\end{align*}
	Now we deal with the influence of intermediate values of $\rho$
	\begin{align*}
		\int_{2z}^{1/2}  K(z, \rho) (-\log \rho)^{\frac 1 m} d \rho
		 & \le \int_{2z}^{1/2}   Q(z/\rho) \rho(-\log \rho)^{\frac 1 m} d \rho
		\le C \int_{2z}^{1/2}(-\log z/\rho)\rho(-\log \rho)^{\frac 1 m} d \rho
		\\
		 & \le C\left( (-\log z) \int_{z}^{\frac 1 2} \rho (-\log \rho)^{\frac 1 m} d\rho -   \int_{z}^{\frac 1 2} \rho (-\log \rho)^{\frac 1 m + 1}d\rho \right)
		\\
		 & \le C (-\log z).
	\end{align*}
	Lastly, we control the terms where $\rho$ is very far from $z$
	\begin{align*}
		\int_{\frac 1 2}^{1}  K(z, \rho) (1-\rho)^{\frac{2-\alpha}{m-1}} d \rho
		 & \le \int_{1/2}^{1}  Q(z/\rho) (1-\rho)^{\frac{2-\alpha}{m-1}} d \rho
		\le C \int_{1/2}^{1} (-\log z/\rho) (1-\rho)^{\frac{2-\alpha}{m-1}} d \rho
		\\
		 &
		\le  C (-\log (z/2)) \int_{1/2}^{1} (1-\rho)^{\frac{2-\alpha}{m-1}} d \rho
		\le C (-\log z).
	\end{align*}
	Since $m > 1$, we conclude that $(\fixedpoint\mathcal V (z))^{m} \le C (-\log z) \le C \mathcal V(z)^m$.
\end{proof}

\subsection{Proof of \Cref{thm:existence} for $m > 1$ and \Cref{thm:slow-diffusion}}
\label{sec:existence proof m>1}
\paragraph{Construction of the canonical profile of support $[0,1]$.}
Taking into account this last equivalence, we can construct the supersolutions
\begin{equation}
	\label{eq:sub and supersolutions}
	\begin{aligned}
		\overline{\mathcal U}(z)  & = \overline c \mathcal V_{\alpha,m}(z) \quad \text{where } \overline c \ge \left(\sup_{z \in (0,1)} \frac{\fixedpoint \mathcal V_{\alpha,m} (z)}{\mathcal V_{\alpha,m}(z)}\right)^{\frac m {m-1}},
		\\
		\underline{\mathcal U}(z) & = \underline c \mathcal V_{\alpha,m}(z) \quad \text{where } \underline c \le \left(\inf_{z \in (0,1)} \frac{\fixedpoint \mathcal V_{\alpha,m} (z)}{\mathcal V_{\alpha,m}(z)}\right)^{\frac m {m-1}}.
	\end{aligned}
\end{equation}
We use the method of sub- and supersolutions using \eqref{eq:sub and supersolutions}. Let $U^{(0)} = \underline{\mathcal U}$ and $U^{(n)} = \fixedpoint U^{(n-1)}$. In every iteration, since $K \ge 0$ we have that
\begin{equation*}
	U^{(n)} \le \fixedpoint U^{(n)} \le \fixedpoint \overline{\mathcal U} \le \overline{\mathcal U}.
\end{equation*}
Therefore, $U^{(n)}$ is a point-wise non-decreasing sequence of continuous functions bounded above by $\overline{\mathcal U}$. Therefore, it has a point-wise limit $\mathcal U$. Since $\overline{\mathcal U}$ provides an integrable bound for the kernel, we can apply the dominated convergence theorem to deduce that $\mathcal U = \fixedpoint \mathcal U$. Finally, since $\mathcal U$ is the image of the integral operator $\fixedpoint$, we immediately deduce $\mathcal U^m \in C^1((0,\infty))$. Due to the sub- and supersolutions, the support of $\mathcal U$ is precisely $[0,1]$. Furthermore, due to the integrability of $z^{d-1}\overline{\mathcal U}(z)$, we have that $z^{d-1} \mathcal U(z) \in L^1 (0,\infty)$. Indeed, for $d\ge3$ the radial density is bounded near zero by $Cz^{d-1-(d-2)/m}$, whose exponent exceeds $-1$; the bounded or logarithmic profiles for $d=1,2$ are likewise integrable with the radial weight.

\paragraph{Construction of a solution of given mass.}
Using \Cref{sec:scaling} we construct the solution of mass $M$ by \eqref{eq:rescaling profile to correct mass}. The uniqueness of $U_{\alpha,m,M}$ is due to \Cref{thm:comparison principle}. Indeed, any distinct profile of the same mass intersects this one at a positive value and is therefore compactly supported by \Cref{thm:profile forward uniqueness}. Ordering by mass and $U_{\alpha,m,M}(z)\to\infty$ as $M\to\infty$ follow from \eqref{eq:rescaling profile to correct mass} and the positivity and monotonicity of $\mathcal U_{\alpha,m}$.

\paragraph{The solution satisfies the weak formulation}
	We verify the regularity hypotheses of \Cref{lem:sss is weak solution}. First take the canonical profile $U=\mathcal U_{\alpha,m}$, whose support is $[0,1]$, and put $\gamma=(2-\alpha)/(m-1)>0$. The global comparison functions already give
	\[
		c(1-z)^\gamma\le U(z)\le C(1-z)^\gamma
		\qquad\text{for }z\text{ close to }1.
	\]
	Using \eqref{intdiff. equation} and \eqref{eq:Q'}, we obtain
	\[
		0\le-(U^m)'(z)
		\le C\int_z^1(\rho-z)^{-\alpha}(1-\rho)^\gamma\,d\rho
		\le C(1-z)^{\gamma+1-\alpha}.
	\]
	In particular, $(U^m)'(z)\to0$ as $z\to1^-$, and the extension by zero has continuous spatial flux at the free boundary. The lower comparison bound also yields
	\[
		|U'(z)|=\frac{|(U^m)'(z)|}{mU(z)^{m-1}}
		\le C(1-z)^{\gamma-1},
	\]
	which is integrable at one. Thus the zero extension of $U$ is locally absolutely continuous away from the origin, including across its support boundary.

	For the function $F$ defined in the profile derivation, these bounds imply $F\in C^1(0,1)$. Indeed, on a compact interval of $z\in(0,1)$, split the $\sigma$-integral at a fixed $c<1$ above all the values $z^{1/b}$. On $(c,1)$, $U'$ is uniformly bounded and the factor $(1-\sigma)^{-\alpha}$ is integrable. On $(0,c)$, change variables to $\rho=z\sigma^{-b}$; the kernel is smooth away from its diagonal, and the integrable boundary behaviour just established permits differentiation. Tonelli's theorem applied directly to \eqref{eq:Q'} recovers \eqref{eq:equation for dUm/dz} from the integral profile equation. Differentiating that identity gives $U^m\in C^2(0,1)$ and \eqref{eq:profile equation second order}. The associated self-similar function is therefore a classical solution in $S$; at each fixed $x\ne0$, its time dependence is absolutely continuous across the arrival time by the bound on $U'$.

	These properties are preserved by the mass rescaling. Finite mass and \Cref{lem:limit of flux holds} supply the remaining hypotheses.
We can apply \Cref{lem:sss is weak solution}.

\paragraph{Sharp constant at the free boundary.}
\label{sec:sharp constant at free boundary}
Set $h(z)=(1-z)_+^{(2-\alpha)/(m-1)}$. The global barriers bound $\mathcal U/h$ above and below by positive constants on every $[z_0,1)$.
Similarly to \eqref{eq:sub and supersolutions}, we can use sub- and supersolutions in $(z,1)$ to deduce that
\begin{equation}
	\label{eq:slow sub and supersolutions near free boundary}
	\left(\inf_{\xi \in (z,1)} \frac{\int_\xi^1 K(\xi, \rho) h(\rho) d\rho}{h(\xi)^m}\right)^{\frac 1 {m-1}} h(z) \le \mathcal U(z) \le \left(\sup_{\xi \in (z,1)} \frac{\int_\xi^1 K(\xi, \rho) h(\rho) d\rho}{h(\xi)^m}\right)^{\frac 1 {m-1}} h(z).
\end{equation}
Indeed, these multiples of $h$ are sub- and supersolutions on $[z,\infty)$; the bounded ratios allow us to apply \Cref{lem:comparison by homogeneity} in both directions.
To compute this limit directly, we set $\gamma = \frac{2-\alpha}{m-1}$, which implies $\gamma m = 2 - \alpha + \gamma$. We evaluate the integral using the change of variables $1-\rho = (1-z)\sigma$, reversing the integration limits:
\begin{equation}
	\label{eq:slow sharp constant at free boundary deduction}
	\begin{aligned}
		\lim_{z \to 1^-} \frac{\int_z^1 K(z, \rho) h(\rho) d\rho} {h(z)^m}
		 & = \frac{b^\alpha}{\Gamma(2-\alpha)} \lim_{z \to 1^-} \frac{1}{(1-z)^{\gamma m}} \int_{z}^1 \rho(1-\tfrac{z}{\rho})^{1-\alpha} (1-\rho)^\gamma d \rho
		\\
		 & = \frac{b^\alpha}{\Gamma(2-\alpha)} \lim_{z \to 1^-} \frac{(1-z)^{2-\alpha+\gamma}}{(1-z)^{\gamma m}} \int_{0}^1 [1-(1-z)\sigma]^\alpha(1-\sigma)^{1-\alpha} \sigma^\gamma d \sigma
		\\
		 & =\frac{b^\alpha}{\Gamma(2-\alpha)} \int_{0}^1 (1-\sigma)^{1-\alpha} \sigma^\gamma d \sigma = \frac{b^\alpha}{\Gamma(2-\alpha)} B(2-\alpha, 1 + \gamma)                                              \\
		 & = \frac{b^\alpha \Gamma(1 + \gamma)}{\Gamma(3-\alpha+\gamma)},
	\end{aligned}
\end{equation}
since $\rho \to 1$ as $z \to 1^-$, and the powers of $(1-z)$ cancel because $\gamma m = 2-\alpha+\gamma$. Dominated convergence applies to the integrable beta density. Evaluating the final Beta function integral, we recover the exact value for the sharp constant \eqref{eq:slow sharp constant at free boundary}.

\paragraph{Classical limit as $\alpha \to 1^-$.}
Fix $m>1$ and $z_0\in(0,1)$. For $\alpha$ close to one, $b_{\alpha,m}$ stays bounded above and below by positive constants. The mean value theorem in \eqref{eq:Q} and the beta integral above give, uniformly in $\alpha$,
	\[
		c(z_0)(\rho-z)^{1-\alpha}\le K_{\alpha,m}(z,\rho)\le C(z_0)(\rho-z)^{1-\alpha},
		\qquad z_0\le z<\rho\le1,
	\]
	\[
		c_1(z_0)(1-z)^{(2-\alpha)/(m-1)}\le\mathcal U_{\alpha,m}(z)
		\le C_1(z_0)(1-z)^{(2-\alpha)/(m-1)},\qquad z_0\le z\le1.
	\]
	Here $1-\sigma^{1/b}$ is uniformly comparable to $1-\sigma$ on $[z_0,1]$, and the factors $1/\Gamma(2-\alpha)$ and $B(2-\alpha,1+(2-\alpha)/(m-1))$ remain bounded and positive. Also, \eqref{intdiff. equation} gives
	\[
		\left|(\mathcal U_{\alpha,m}^m)'(z)\right|
		\le\frac{C(z_0)}{\Gamma(1-\alpha)}\int_z^1(\rho-z)^{-\alpha}\,d\rho
		=\frac{C(z_0)(1-z)^{1-\alpha}}{\Gamma(2-\alpha)}\le C'(z_0).
	\]
	Thus the profiles are uniformly H\"older continuous with exponent $1/m$ on $[z_0,1]$. By Arzel\`a--Ascoli, we have that $\{\mathcal U_{\alpha,m} : \alpha \text{ close to } 1 \}$ is pre-compact in $C_{loc}((0,1])$.
Hence, there is a subsequence $\mathcal U_{\alpha_k,m}$ that converges to $\widehat{\mathcal U}_m$. Next we observe from the sub- and supersolutions that $\mathcal U_{\alpha,m,d}$ are uniformly controlled above and below locally on $(0,1)$ as $\alpha \to 1^-$. So, the support of the limit is still $[0,1]$. Moving to the limit of the equation by dominated convergence, using the kernel bounds and \Cref{eq:Q limit alpha -> 1}, we deduce that $\widehat{\mathcal U}_m$ is a solution of the equation at $\alpha = 1$. More precisely,
		\[
			\widehat{\mathcal U}_m(z)^m=b_{1,m}\int_z^1\rho\widehat{\mathcal U}_m(\rho)\,d\rho,
			\qquad
			(\widehat{\mathcal U}_m^{m-1})'=-\frac{m-1}{m}b_{1,m}z.
		\]
		The value at one determines \eqref{eq:m>1 limit alpha->1}. Lastly, we use uniqueness to show that it is $\mathcal U_{m}$ and that the whole family converges.

\paragraph{Mesa limit.} Lastly, we discuss the limit $m \to \infty$.
\subparagraph{Mesa limit of $\mathcal U_{\alpha,m}$.}
For $z > 0$ we go back to the estimate near the free boundary \eqref{eq:slow sub and supersolutions near free boundary}.
	Fix $\alpha\in(0,1)$ and take $m\ge2$, so $b<1$. Convexity gives
		\[
			1-\sigma\le1-\sigma^{1/b}\le b^{-1}(1-\sigma),\qquad 0<\sigma<1.
		\]
		Consequently, \eqref{eq:Q} implies
		\[
			c(z_0,\alpha,d)b^\alpha(\rho-z)^{1-\alpha}\le K_{\alpha,m}(z,\rho)
			\le C(z_0,\alpha,d)(\rho-z)^{1-\alpha},\qquad z_0\le z<\rho\le1.
		\]
		The beta factor $B(2-\alpha,1+(2-\alpha)/(m-1))$ stays bounded above and below by positive constants.
Using the explicit form of $Q_m'$ we can deduce similarly to \eqref{eq:slow sharp constant at free boundary deduction} that
\begin{equation}
	\label{eq:slow upper and lower bounds near z=1}
	(b^{\alpha} c(d,\alpha,z_0))^{\frac{1}{m-1}} (1-z)_+^{\frac{2-\alpha}{m-1}} \le \mathcal U_{\alpha,m}(z) \le (C(d,\alpha,z_0))^{\frac{1}{m-1}} (1-z)_+^{\frac{2-\alpha}{m-1}},
\end{equation}
for $z_0\le z\le1$, with constants independent of $m$.
Given that $b_{\alpha,m}^{\frac 1 {m-1}} = \left(\frac{\alpha}{2+d(m-1)}\right)^{\frac 1{m-1}} \to 1$ as $m \to \infty$ we deduce that
\begin{equation*}
	\lim_{m\to \infty} \mathcal U_{\alpha,m} (z) =
	\begin{dcases}
		1 & \text{if } 0 < z < 1, \\
		0 & \text{if } z \ge 1.
	\end{dcases}
\end{equation*}

\subparagraph{Mesa limit of $U_{\alpha,m,M}$.} Now we look at the construction of the profile of mass $M$ deduced in \eqref{eq:rescaling profile to correct mass} and let $L = A^{-\frac{m-1}{2}}$.
		For $m\ge2$, set $\mathcal H_d(z)=1+\int_z^1s^{1-d}\,ds$ for $0<z<1$. By \eqref{eq:Q}, $1-\sigma^{1/b}\ge1-\sigma$, and monotonicity,
		\[
			\mathcal U_{\alpha,m}(z)^{m-1}
			\le\int_z^1K_{\alpha,m}(z,\rho)\,d\rho
			\le C(\alpha,d)\mathcal H_d(z).
		\]
		The last bound follows by splitting the defining integral for $Q$ at $1/2$. Thus $\mathcal U_{\alpha,m}\le1+C(\alpha,d)\mathcal H_d$, uniformly in $m$, and this bound is integrable against $z^{d-1}dz$.
Due to the point-wise convergence and the existence of this uniform integrable bound, we can use the dominated convergence theorem to deduce that
	the canonical mass tends to $|B_1|$. Formula \eqref{eq:rescaling profile to correct mass} therefore gives
\begin{equation*}
	A\to1,\qquad
	L(\alpha,m,d,M)\to \left(\frac{|B_1|}{M}\right)^{1/d}=R_M^{-1}.
\end{equation*}

\subparagraph{Mesa limit of $u_{\alpha,m,M}$.} Now we notice that $b_{\alpha,m} \to 0$, therefore we have that at each fixed $t>0$
\begin{equation*}
	u_{\alpha,m,M}(t,x) = t^{-b_{\alpha,m} d} U_{\alpha,m,M}(t^{-b_{\alpha,m}}|x|) \to \chi_{B_{R_M}}(x),\qquad |x|\notin\{0,R_M\}.
\end{equation*}
Monotonicity gives convergence at the moving arguments away from the origin and the interface. Since both sides have mass $M$, dominated convergence applied to their pointwise minimum also gives convergence in $L^1(\mathbb R^d)$.
This concludes the proof. \qed

\begin{remark}
	We point out that the stability of \eqref{eq:slow sharp constant at free boundary} is consistent with the limit as $\alpha \to 1^-$.
	We have that
	\begin{equation*}
		\lim_{\alpha \to 1^-} \frac{\Gamma(1 + \tfrac{2-\alpha}{m-1} )}{\Gamma(3-\alpha+\tfrac{2-\alpha}{m-1})} = \frac{\Gamma(1 + \frac{1}{m-1})}{\Gamma(2 + \frac{1}{m-1})} = \frac{1}{1 + \frac{1}{m-1}} = \frac{(m-1)}{m},
	\end{equation*}
	and hence for $z \sim 1$
	\begin{align*}
		\mathcal U_{\alpha,m}(z) & \sim \left( \frac{b^\alpha \Gamma(1 + \tfrac{2-\alpha}{m-1} )}{\Gamma(3-\alpha+\tfrac{2-\alpha}{m-1})} \right)^{\frac{1}{m-1}} (1-z)_+^{\frac{2-\alpha}{m-1}}
		\\
		                         & \sim \left( \frac{b_{\alpha,m}^\alpha \Gamma(1 + \tfrac{2-\alpha}{m-1} )}{\Gamma(3-\alpha+\tfrac{2-\alpha}{m-1})} \right)^{\frac{1}{m-1}} \left(\frac{(1-z^2)_+}{2}\right)^{\frac{2-\alpha}{m-1}}.
	\end{align*}
	The last expression tends, as $\alpha\to1^-$, to
	\[
		\left( \frac{b_{1,m}(m-1)}{2m} (1-z^2)_+ \right)^{\frac{1}{m-1}} = \mathcal U_m(z).
	\]
\end{remark}

\subsection{Numerical analysis for $m > 1$}

We now restrict ourselves to piece-wise constant functions. Then we look for the fixed point $U = \mathcal K_h U$ where
\begin{equation}
	(\fixedpoint_h U)_i \coloneqq  \left( \sum_{j=i}^{+\infty} U_j K_{ij} \right)^{\frac 1 m}, \quad \text{ where }
	K_{ij} = \int_{z_{j}}^{z_{j+1}} K(z_i, \rho) d\rho.
\end{equation}
In dimension $d = 1$ we can use the closed formula for $Q$ given by \eqref{eq:Q from incomplete gammas}. The complementary form below is stable as $\alpha \to 1^-$. In every dimension, Fubini's theorem gives an equivalent formula. Put
		\[
			a_{ij}=\frac{z_i}{z_{j+1}},\qquad
			D_{ij}(\eta)=\frac{z_{j+1}^2-\max\{z_j,z_i/\eta\}^2}{2},
			\qquad
			r_b(\eta)=\frac{1-\eta^{1/b}}{1-\eta},\quad r_b(1)=\frac1b.
		\]
		For $0<z_i\le z_j<z_{j+1}$, use $\eta(t)=1-(1-a_{ij})t^{1/(1-\alpha)}$ to obtain
\begin{equation}
	\begin{aligned}
		K_{ij} & =\int_{z_j}^{z_{j+1}} \rho Q(z_i/\rho)\,d\rho
			                                                   = -\int_{z_j}^{z_{j+1}}\int_{z_i/\rho}^1 \rho Q'(\eta)\,d\eta\,d\rho \\
		       & =-\int_{\frac{z_i}{z_{j+1}}}^1 Q'(\eta)\frac{z_{j+1}^2-\max\{z_j,\frac{z_i}{\eta}\}^2}{2}\,d\eta                \\
		       & =\frac{(1-a_{ij})^{1-\alpha}}{\Gamma(2-\alpha)}
			         \int_0^1 r_b(\eta(t))^{-\alpha}\eta(t)^{1-d}D_{ij}(\eta(t))\,dt.
	\end{aligned}
\end{equation}
The integrand has a finite limit at $t=0$. Splitting at the corner corresponding to $\eta=z_i/z_j$ permits accurate quadrature without truncating the singular endpoint.
In the limit $\alpha \to 1^-$ we recover
\[
	K_{ij}\longrightarrow b_{1,m}\frac{z_{j+1}^2-z_j^2}{2}
	.
\]
We have implemented this scheme in \texttt{julia} using the \texttt{Integrals.jl} package to compute these integrals, using the Gauss-Kronrod quadrature implemented in \texttt{QuadGK.jl}.
	We use the exact weights when $\alpha=1$. When evaluating the one-dimensional closed formula, we use the complementary regularized beta function
		\[
			Q(\eta)=\frac{\Gamma(b+1)}{\Gamma(b+1-\alpha)}
			\bigl[1-I_{\eta^{1/b}}(b,1-\alpha)\bigr],
		\]
		evaluated directly as a complementary function to avoid cancellation.
In $d=1$ our scheme is similar to \cite{Plociniczak2019}.

\bigskip

We find approximate fixed points this problem by Picard iterations $U^{(n)} = \fixedpoint_h U^{(n-1)}$. We start iterating from a sub-solution we construct as follows.
We consider a mesh with $z_0 = 0$ and $z_I = 1$. As a initial datum we take $U^{(0)}_i = c \mathcal V_{\alpha,m}(z_i)$ with $c$ small enough
\begin{equation*}
	c \le \left( \min_{i=1,\cdots,I-1} \frac{\sum_{j=i}^{I-1}  K_{ij} \mathcal V_{\alpha,m}(z_j) }{\mathcal V_{\alpha,m}(z_i)^m}\right)^{\frac 1 {m-1}}
\end{equation*}
so that $U^{(0)} \le \fixedpoint_h U^{(0)}$. This ensures that $U^{(n)} = \fixedpoint_h U^{(n-1)}$ is non-decreasing in $n$. Conversely, we can construct a supersolution by taking $\overline U_i = C \mathcal V_{\alpha,m}(z_i)$ with $C$ chosen large enough:
\begin{equation*}
	C \ge \left( \max_{i=1,\cdots,I-1} \frac{\sum_{j=i}^{I-1}  K_{ij} \mathcal V_{\alpha,m}(z_j) }{\mathcal V_{\alpha,m}(z_i)^m}\right)^{\frac 1 {m-1}}.
\end{equation*}
This ensures that $U^{(n)} \le \overline U_i$.

\begin{remark}
	From a theoretical point of view, if $\mathcal U$ is a subsolution to \eqref{eq:profile equation}, then we can take $U_i^{(0)} = \underline{\mathcal U}(z_i)$ with a non-increasing sub-solution we have that
	\begin{align*}
		\underline{\mathcal U}(z_i)^m & \le  \int_{z_i}^{+\infty} K(z,\rho) \underline{\mathcal U}(\rho) d \rho   =  \sum_{j=i}^{+\infty}  \int_{z_j}^{z_{j+1}} K(z,\rho) \underline{\mathcal U}(\rho) d \rho
		\le \sum_{j=i}^{+\infty} \underline{\mathcal U}(z_j) K_{ij} .
	\end{align*}
	Therefore, the sequence $U^{(n+1)}=\fixedpoint_h U^{(n)}$ is monotone increasing. This sampling argument does not provide discrete supersolutions.
\end{remark}

\begin{remark}
	Convergence of this scheme as $h = \max_i |z_{i+1} - z_{i}| \to 0$ tends to $0$ can be accomplished proving Hölder estimates of the numerical solutions.
\end{remark}

\section{Fast diffusion $m_c < m < 1$}

\subsection{Very singular solution}
Let us define the possibly infinite integral
\[
	\mathsf Q(\gamma) \coloneqq \int_1^\infty Q(\sigma^{-1}) \sigma^{1-\gamma} d \sigma = \int_0^1 Q(\sigma) \sigma^{\gamma - 3} d \sigma.
\]
When $\gamma > \max\{2,d\}$, the boundary terms vanish, and we can integrate by parts to write:
\begin{equation}
	\label{eq:mathsf Q}
	\mathsf Q(\gamma) = \int_0^1 (-Q'(\sigma)) \frac{\sigma^{\gamma - 2}}{\gamma - 2} d \sigma = \frac{b}{\gamma - 2} \frac{B\Big( 1-\alpha, b(\gamma-d) \Big)}{\Gamma(1-\alpha)}= \frac{b}{\gamma - 2} \frac{\Gamma\Big( b(\gamma-d) \Big)}{\Gamma(1-\alpha + b(\gamma-d))}.
\end{equation}
Notice that if $m > m_c = \frac{(d-2)_{+}}{d}$ and $m<1$, then $\gamma = \frac{2}{1-m} > \max\{2,d\}$.

\begin{lemma}
	\label{lem:vss}
	For $m \in (m_c,1)$ we have that \eqref{eq:vss} is a solution to \eqref{eq:profile equation}.
\end{lemma}
\begin{proof}
	Consider a candidate of the form $U(z) = c z^{-\gamma}$. Then we have $z > 0$ (and putting $\rho=z/\sigma$)
	\begin{align*}
		(\fixedpoint U(z))^m & = c \int_z^\infty \rho Q(z/\rho)\, {\rho ^{-\gamma}} \,d \rho       \\
		                     & =   c z^{2-\gamma} {\int_0^1 Q(\sigma) \sigma^{\gamma-3} d \sigma } \\
		                     & = c^{1-m} \int_0^1 Q(\sigma) \sigma^{\gamma-3} d \sigma
		\, (c z^{\frac{2 - \gamma}m})^m.
	\end{align*}
	The correct power is therefore a $\gamma$ such that $-\gamma = \frac{2-\gamma}{m}$, that is, $\gamma = \frac{2}{1-m}$. Note that $\gamma>d$ exactly for $m> m_c=(d-2)_{+}/d$.
		By \eqref{eq:mathsf Q}, the integral is finite for $\gamma>\max\{2,d\}$.
	We get an exact solution for the precise value
	\[
		(c^*)^{m-1}= \int_0^1 Q(\sigma) \sigma^{\gamma-3} d \sigma= \mathsf Q(\tfrac{2}{1-m}).
	\]
	This concludes the proof.
\end{proof}

\subsection{Higher order expansion of the tails of solutions below \eqref{eq:vss}}

This is a decreasing and positive function on $(\max\{2,d\},\infty)$.
By the dominated convergence theorem, we have $\mathsf Q(+\infty) = 0$. We have calculated $(c^*)^{m-1} = \mathsf Q(\frac{2}{1-m})$ and now we look for the first order expansion
\[
	U(z) = c^* z^{-\frac 2 {1-m}} ( 1 + A z^{-\gamma} ).
\]
We directly compute that if $1 + Az^{-\gamma} \ge 0$ then as $z\rightarrow \infty$, we have
\begin{align*}
	\frac{U(z)^m}{U^*(z)^m}                & = (1 + Az^{-\gamma})^m = 1 + m A z ^{-\gamma} + {{O(A^2z^{-2\gamma})}},
	\\
	\frac{(\fixedpoint U(z))^m}{U^* (z)^m} & = 1 + \frac{\mathsf Q(\frac{2}{1-m}+\gamma)}{\mathsf Q(\frac{2}{1-m})} A z^{-\gamma}.
\end{align*}
where $U^*$ is the VSS given by \eqref{eq:vss}.
Therefore, since $m \in (0,1)$ we can take $\gamma^*$ such that
\begin{equation}
	\label{eq:gamma^star}
	\mathsf Q\left(\tfrac{2}{1-m}+\gamma^*\right) = m \mathsf Q\left(\tfrac{2}{1-m}\right).
\end{equation}
And the constant $A$ is free.

\begin{remark}
	We recall the explicit formula \eqref{eq:mathsf Q} in terms of functions allows to compute efficiently $\gamma^*$ using numerical methods.
\end{remark}

\paragraph{A sub-solution.}
To strictly estimate the error, recall that for $m \in (0,1)$, the map $s \mapsto s^m$ is strictly concave. Therefore, its tangent line at $s=1$ lies strictly above the graph, which gives the elementary inequality $(1-x)^m \le 1 - mx$ for all $x \in (0,1)$.
Setting $x = A z^{-\gamma^*}$ with $A > 0$, we immediately have:
\begin{equation*}
	\frac{U(z)^m}{U^*(z)^m} = (1 - A z^{-\gamma^*})^m \le 1 - m A z^{-\gamma^*} = \frac{(\fixedpoint U(z))^m}{U^* (z)^m}.
\end{equation*}
Also, where $U(z) \le 0$, the positive part $(U(z))_+ = 0$ is trivially a subsolution. As a consequence, we have the following result.
	For $A<0$, the same conclusion follows from the tangent inequality $(1-x)^m\le1-mx$ for $x<0$.

\begin{lemma}
	For $m \in (m_c,1)$ and $A\in \mathbb R$ we find that $\underline{\mathcal U}(z) = U^*(z) ( 1 - A z^{-\gamma^*})_+$ is a subsolution to \eqref{eq:profile equation} for all $z \in (0, +\infty)$.
\end{lemma}

\paragraph{A super-solution for the tails.}
To construct a supersolution, let us consider
\begin{equation}
	\label{eq:m < 1 tail supersolution}
	U(z) = U^*(z) \left( 1 - z^{-\gamma^*} + z^{-\frac{3}{2}\gamma^*} \right).
\end{equation}
Now we estimate
\begin{align*}
	\frac{U(z)^m - (\fixedpoint U(z))^m}{U^*(z)^m}
	 & \ge \Lambda z^{-\frac{3}{2}\gamma^*}
	+ R(z)  \quad \text{where } \Lambda =  m - \frac{\mathsf Q(\frac{2}{1-m}+\frac{3}{2}\gamma^*)}{\mathsf Q(\frac{2}{1-m})} \ge 0.
\end{align*}
where $R$ is the Taylor remainder.
\begin{lemma}
	\label{lem:fast supersolution tails}
	There exists $z_0(\alpha, m)$ such that
	\begin{equation}
		\label{eq:fast supersolution tails}
		\overline{\mathcal U}(z) = U^*(z) (1 - z^{-\gamma^*} + z^{-\frac 3 2\gamma^*} ),
	\end{equation}
	is a super-solution for all $z \ge z_0(\alpha,m)$.
\end{lemma}
\begin{proof}
	We point out that $\overline{\mathcal U}(z) / U^*(z) \in [1/2,1]$ for $z \ge z_0 \ge 1$ and therefore
	the Taylor remainder term can be controlled by
	\begin{align*}
		|R(z)| & \le \frac{1}{2}\left( z^{-\gamma^*} - z^{-\frac 3 2 \gamma^*} \right)^2 \sup_{\xi \in [1/2,1]} |m(m-1) \xi^{m-2}| \\
		       & \le \frac{|m(m-1)|}{2}{2^{2-m}}z^{-2\gamma^*}.
	\end{align*}
	Indeed, $1-x+x^{3/2}\in[23/27,1]$ for $x\in[0,1]$, and $\Lambda>0$ by the strict decrease of $\mathsf Q$.
	We choose $z_0$ such that $\frac{|m(m-1)|}{2}2^{2-m}z_0^{-\gamma^*/2}\le\Lambda$; then $|R(z)|\le\Lambda z^{-3\gamma^*/2}$ for every $z\ge z_0$.
\end{proof}

\subsection{Super-solutions close to $z = 0$}
\begin{lemma}
\label{lem:fast supersolution close 0}
Assume $d \ge 2$, $m > 1 - \frac{2}{d}$ with $m<1$ and let
\begin{equation*}
\gamma \in
\begin{dcases}
[\tfrac{d-2}{m}, d) & \text{if } d \ge 3, \\
(0,2) & \text{if } d=2.
\end{dcases}
\end{equation*}
Put $g(\eta)=(1-\eta^{1/b})^{-\alpha}$ for $0\le\eta<1$ and define
	\[
		H_\gamma=\sup_{0<z<1}\int_z^1 g(z/\rho)\rho^{d-\gamma-1}\,d\rho<\infty.
	\]
Let us pick $c>0$ such that
\[
	c^{1-m}\le\frac{m\gamma\Gamma(1-\alpha)}{H_\gamma+1}.
\]
Assume that there exists a nonnegative continuous solution $U$ such that $U(1)\le c$ and
\begin{equation*}
	B_U:=\int_1^\infty (1-\rho^{-\frac 1 b})^{-\alpha} U(\rho) \rho^{d-1} d \rho \le c.
\end{equation*}
Then, we have
\begin{equation*}
	\overline{U}(z)=
	\begin{dcases}
		c z^{-\gamma} & \text{if } z < 1, \\
		U(z)          & \text{if } z\ge1,
	\end{dcases}
\end{equation*}
is a super-solution and, furthermore, $U(z)\le\overline U(z)$ for $z\in(0,+\infty)$.
\end{lemma}
\begin{proof}
	To bound $H_\gamma$, split at $\min\{2z,1\}$. Substituting $\rho=zs$ bounds the first part by
		\[
			z^{d-\gamma}\int_1^2g(1/s)s^{d-\gamma-1}\,ds,
		\]
		which is uniformly finite because $\alpha<1$ and $\gamma<d$. On the remainder, use $g(z/\rho)\le g(1/2)$ and $\int_0^1\rho^{d-\gamma-1}\,d\rho<\infty$.

	For $z\ge1$ we have that $U$ is exactly a solution. For $z\in(0,1)$
	\begin{align*}
		\int_z^\infty & (-Q'(z/\rho))\overline U(\rho)\,d\rho                                                                                                                                                         \\
		              & \le\frac{z^{1-d}}{\Gamma(1-\alpha)}\left(c\int_z^1 (1-(\tfrac z\rho)^{1/b})^{-\alpha}\rho^{-\gamma+d-1}\,d\rho+\int_1^\infty(1-(\tfrac z\rho)^{1/b})^{-\alpha}U(\rho)\rho^{d-1}\,d\rho\right) \\
		              & \le\frac{z^{1-d}}{\Gamma(1-\alpha)}(cH_\gamma+B_U)
			                \le m\gamma c^m z^{1-d}
			                \le m\gamma c^m z^{-m\gamma-1}                                                                                                                              \\
		              & =-(\overline U^m)'(z).
	\end{align*}
	Here $m\gamma\ge d-2$. Since $c^m\ge U(1)^m=(\fixedpoint\overline U(1))^m$, integration from $z$ to $1$ proves the supersolution inequality.
	Since $\underline U=0$ is a sub-solution, we have that $\widetilde U=\lim_{n\to\infty}\fixedpoint^n\overline U$ is a solution to the problem.
		Indeed, these iterates decrease and are bounded by $\overline U$, so dominated convergence gives a continuous solution.
	Given that $\overline U(z)=U(z)$ for $z>1$ we have that $\fixedpoint^n\overline U(z)=U(z)$ for $z>1$, and therefore $\widetilde U(z)=U(z)$ for $z>1$.
	Consider
	\[
		z_0=\inf\{z>0:U=\widetilde U\text{ in }(z,+\infty)\}.
	\]
	Assume that $z_0>0$. For $z<z_0$ we write using the mean value theorem
	\begin{align*}
		|U(z)-\widetilde U(z)|
		 & =\frac{\xi^{\frac1m-1}}m\left|\int_z^{z_0}\rho Q(z/\rho)(U(\rho)-\widetilde U(\rho))\,d\rho\right| \\
		 & \le\left(\sup_{\rho\in[z,z_0]}|U(\rho)-\widetilde U(\rho)|\right)
		\frac{U(z)^{1-m}+\widetilde U(z)^{1-m}}m\int_z^{z_0}\rho Q(z/\rho)\,d\rho.
	\end{align*}
	Here $\xi$ lies between $U(z)^m$ and $\widetilde U(z)^m$, and $1/m-1>0$.
	Now taking supremum again to some $z_1<z_0$ we recover
	\begin{align*}
		\sup_{z\in[z_1,z_0]} & |U(z)-\widetilde U(z)|                          \\
		                     & \le\sup_{z\in[z_1,z_0]}|U(z)-\widetilde U(z)|\,
		\sup_{z\in[z_1,z_0]}\frac{U(z)^{1-m}+\widetilde U(z)^{1-m}}m\int_z^{z_0}\rho Q(z/\rho)\,d\rho.
	\end{align*}
	Notice that
	\begin{equation*}
		\lim_{z_1\to z_0^-}\sup_{z\in[z_1,z_0]}\frac{U(z)^{1-m}+\widetilde U(z)^{1-m}}m\int_z^{z_0}\rho Q(z/\rho)\,d\rho=0.
	\end{equation*}
	For $z_1<z_0$ but close enough to $z_0$ this is smaller than $1$, and hence we deduce that
	\[
		\sup_{z\in[z_1,z_0]}|U(z)-\widetilde U(z)|=0.
	\]
	But this is a contradiction to the definition of $z_0$.
\end{proof}

\subsection{Proof of \Cref{thm:existence} for $m \in (m_c,1)$ and \Cref{thm:fast-diffusion}}
\label{sec:existence proof m < 1}
\paragraph{Existence of a family of profiles.}
Using the subsolution $\underline{\mathcal U}(z) = c^*z^{-\frac{2}{1-m}}(1 - z^{-\gamma^*})_+$ and $\overline{\mathcal U} = U^* $ we deduce that
\begin{equation}
	\label{eq:m<1 canonical solution}
	\mathcal U = \lim_n \fixedpoint^n \underline{\mathcal U},
\end{equation}
is in $C((0,+\infty))$ and it is a solution to the problem.
Furthermore, using \eqref{eq:fast supersolution tails} we deduce that $\mathcal U \ne U^*$ and we have the estimate
\begin{equation*}\label{Lbeh.VSSfan}
	U^*(z) ( 1 - z^{-\gamma^*} )_+
	\le \mathcal U(z) \le
	U^*(z) (1 - z^{-\gamma^*} + z^{-\frac 3 2 \gamma^*}), \quad \text{for all } z \ge z_0 > 0.
\end{equation*}
Let us define
\[
	U_R(z) \coloneqq R^{\frac{-2}{1-m}} \mathcal U\left(\frac{z}{R}\right).
\]
\begin{lemma}
	\label{lem:fast rescaled limits in R}
	For $m \in (m_c,1)$ and $z>0$ we have that $U_R(z)$ is a decreasing family in $R$ and, furthermore, for each $z>0$,
	\[
		U_R(z) \to
		\begin{dcases}
			U^*(z) & \text{as } R \to 0,      \\
			0      & \text{as } R \to \infty.
		\end{dcases}
	\]
\end{lemma}
\begin{proof}
	Notice that
	\[
		U_R(z) = c^*z^{-\frac 2{1-m}}\Big(1 - R^{\gamma^*} z^{-\gamma^*} + o\left(R^{\gamma^*} z^{-\gamma^*}\right)\Big).
	\]
	Hence, for $z$ fixed, we have
	\[
		U_R(z) \to U^*(z) \text{ as } R \to 0.
	\]
	Furthermore, if $\overline R > \underline R$ then we have that
	$U_{\overline R}(z) < U_{\underline R}(z)$ for $z$ sufficiently large. Consider
	\begin{equation*}
		z_1 = \inf\{Z > 0 : \text{for all } z \in (Z,+\infty) \text{ we have that } U_{\overline R}(z) < U_{\underline R}(z)\}.
	\end{equation*}
	Assume that $z_1 > 0$. Then $U_{\overline R}(z_1) = U_{\underline R}(z_1)$. However, we have
	\[
		U_{\overline R}(z_1)^m  = \int_{z_1}^\infty K(z_1,\rho) U_{\overline R}(\rho) d\rho < \int_{z_1}^\infty K(z_1,\rho) U_{\underline R}(\rho) d\rho = U_{\underline R} (z_1)^m ,
	\]
	which is a contradiction. There $U_R(z)$ is decreasing in $R$ for all $z>0$. Now consider
	\[
		U_\infty(z) = \lim_{R \to \infty} U_R (z).
	\]
	Due to the dominated convergence theorem (using $U_R\le U^*$ and the profile equation for $U^*$), $U_\infty(z)=\fixedpoint U_\infty(z)$ for each $z>0$. Furthermore, we have
	\begin{align*}
		U_\infty(z) & = \inf_{R > 0} U_R(z) = \inf_{R>0} R^{-\frac{2}{1-m}} \mathcal U(z/R) =  \inf_{R>0} U^*(z) \frac{\mathcal U(z/R)}{U^*(z/R)} = U^*(z) \inf_{R>0} \frac{\mathcal U(z/R)}{U^*(z/R)} \\
		            & = U^*(z) \inf_{y > 0} \frac{\mathcal U(y)}{U^*(y)} = c_\infty U^*(z).
	\end{align*}
	Since $\mathcal U < U^*$ we have $c_\infty < 1$. Because $U_\infty$ satisfies $U_\infty = \fixedpoint U_\infty$, we have $c_\infty^m U^*(z)^m = c_\infty U^*(z)^m$, which requires $c_\infty^m = c_\infty$. Since $m \in (0,1)$ and $c_\infty < 1$, we deduce that $c_\infty = 0$.
\end{proof}

\paragraph{For $d=1$ solutions are bounded.}
Notice that since $U_R$ is decreasing we can use \eqref{eq:Q'} with $d=1$ and observe that $-Q'$ is increasing.
	Splitting at $\min\{2z,1\}$ as in the preceding lemma gives
		\[
			A:=\sup_{0<z<1}\int_z^1(-Q')(z/\rho)\,d\rho<\infty,
			\qquad B_R:=\int_1^\infty(-Q')(1/\rho)U_R(\rho)\,d\rho.
		\]
Hence for $z\in(0,1)$ we have that
\begin{align*}
	-(U_R^m)'(z) & =\int_z^\infty(-Q')(\tfrac z\rho)U_R(\rho)\,d\rho                                             \\
	             & \le U_R(z)\int_z^1(-Q')(\tfrac z\rho)\,d\rho+\int_1^\infty(-Q')(\tfrac z\rho)U_R(\rho)\,d\rho \\
	             & \le U_R(z)A+B_R.
\end{align*}
Due to \Cref{lem:fast rescaled limits in R} and the dominated convergence theorem $B_R\to0$ as $R\to\infty$.
	Here $(-Q')(1/\rho)U^*(\rho)\in L^1(1,\infty)$ provides a majorant.
We also have $U_R(1)\to0$. Let us define $W(s)=[U_R(1-s)]^m$. Then we have
\begin{equation*}
	W'(s)\le A W(s)^{\frac1m}+B_R,\quad W(0)=U_R(1)^m.
\end{equation*}
Integrating this differential inequality, we see that if $B_R$ and $U_R(1)$ are small enough, then $U_{R}(0)^m=W(1)<\infty$.
Since $U_R$ is bounded for $R$ large enough, then $\mathcal U$ is bounded. Due to tail estimates, $\mathcal U$ has finite mass.

\paragraph{For $d \ge 2$ solutions have finite mass.}
We can apply \Cref{lem:fast supersolution close 0} to $U_{R(c)}$ with $R(c)$ large enough, since $U_R(1),B_{U_R}\to0$ by \Cref{lem:fast rescaled limits in R} and dominated convergence. Since $U_{R(c)}\le c z^{-\gamma}$ near zero for some $\gamma<d$, we conclude using the tail estimates that $U_R$ has finite mass for all $R>0$. And hence so does $\mathcal U$.

\paragraph{Conclusion for $\alpha$ fixed.}
Since we have proven that solutions have finite mass, we can apply \eqref{eq:limit of flux at z=0} to obtain their sharp behavior at $z=0$. Through \eqref{eq:equation for dUm/dz} we deduce that $U^m\in C^1(0,+\infty)$.
	For $q=2/(1-m)>d$, \eqref{intdiff. equation} and $U\le U^*$ give $|U'(z)|\le q c^*_{\alpha,m}z^{-q-1}$. The differentiated integrand of $F$ is bounded locally uniformly by $C_z\sigma^{b(q-d)}(1-\sigma)^{-\alpha}\in L^1(0,1)$. Thus $F\in C^1$, and \eqref{eq:equation for dUm/dz} gives $U^m\in C^2$ and \eqref{eq:profile equation second order}. Time absolute continuity follows from monotonicity and finite mass as in the proof of \Cref{lem:sss is weak solution},
and hence we can apply \Cref{lem:sss is weak solution} to deduce that they are weak solutions.

\paragraph{Limit as $\alpha \to 1^-$.}
We point out that $c^*_{\alpha,m}$ is continuous at $\alpha=1$, and hence $U^*_{\alpha,m}$ is continuous at $\alpha=1$.
Due to \eqref{eq:mathsf Q} we observe that if $\gamma>\max\{2,d\}$
\begin{equation*}
	\lim_{\alpha\to1^-}\mathsf Q_{\alpha,m}(\gamma)=\frac{b_{1,m}}{\gamma-2}.
\end{equation*}
Hence, we can pass $\gamma^*_{\alpha,m}$ to the limit as $\alpha\to1$ in \eqref{eq:gamma^star}, obtaining $\gamma^*_{\alpha,m}\to2$.
And thus we can pass to the limit $\underline{\mathcal U}_{\alpha,m}$ and $\overline{\mathcal U}_{\alpha,m}$. Following the choice in \Cref{lem:fast supersolution tails} we see that $z_0(\alpha,m)$ is bounded as $\alpha\to1^-$:
with $q=2/(1-m)$, we have $\Lambda_{\alpha,m}\to m/(q+1)>0$.
Using $\mathcal U_{\alpha,m}\le U^*_{\alpha,m}$ and the preceding derivative bound, uniform since $c^*_{\alpha,m}\to c^*_{1,m}$, we have that $\mathcal U_{\alpha,m}$ is precompact in $C_{loc}((0,+\infty))$. Any possible limit $\widehat{\mathcal U}$ satisfies
\begin{equation}
	\label{eq:m < 1 alpha = 1 upper and lower bounds}
	U^*_{1,m}(z)(1-z^{-2})_+\le\widehat{\mathcal U}(z)\le U^*_{1,m}(z)(1-z^{-2}+z^{-3})
\end{equation}
for $z$ large enough.
Furthermore, using super-solutions and the dominated convergence theorem, it is a solution to \eqref{eq:profile equation} for $\alpha=1$.
	Indeed, the splitting used in \eqref{eq:Q close to 0} gives $Q_{\alpha,m}(\eta)\le C(1+\int_\eta^1s^{1-d}\,ds)$, uniformly for $\alpha$ close to one. Together with $\mathcal U_{\alpha,m}(\rho)\le C\rho^{-q}$ and $q>\max\{2,d\}$, this provides an integrable majorant for each fixed $z>0$.
Since $\widehat{\mathcal U}$ is a solution of \eqref{eq:profile equation} with $K_{1,m}(z,\rho)=b\rho$ such that \eqref{eq:m < 1 alpha = 1 upper and lower bounds} we can take one derivative and prove that
\begin{equation*}
	-(\widehat{\mathcal U}^m)'(z)=bz\widehat{\mathcal U}(z).
\end{equation*}
Integrating this equation, we deduce that
\begin{equation*}
	\widehat{\mathcal U}(z)=c^*_{1,m}(C+z^2)^{-\frac1{1-m}}.
\end{equation*}
Due to the asymptotic expansion, we conclude that $C=1-m$. Since every convergent sub-sequence has the same limit, the whole sequence $\mathcal U_{\alpha,m}$ converges, and this is its limit.
\qed

\section{Limit $m \to 1$. Proof of \Cref{thm:limit m->1}}
Like in the case $\alpha = 1$, we expect continuity with respect to $m$ and $U_{\alpha,m,M}$ with respect to $m$. Already in this case, the limits are tricky to compute. We have some preliminary results.
	Fix $\alpha\in(0,1)$ and $M>0$, and restrict $m$ to $[m_-,m_+]$ with $m_c<m_-<1<m_+$.
Due to \eqref{eq:Lieb upper bound}  and \eqref{eq:equicontinuity} we have that $\{U_{\alpha,m,M}^m:m\in[m_-,m_+]\setminus\{1\}\}$ is pre-compact in $C_{loc}((0,+\infty))$.
Consider $\widehat U$ the limit of any convergent sub-sequence with $m_j \to 1$. Writing $V_j=U_{\alpha,m_j,M}$, both $V_j^{m_j}$ and $V_j$ converge locally uniformly to $\widehat U$.
Also, by Fatou's lemma we have that
\begin{equation*}
	\widehat M = |\partial B_1| \int_0^\infty \widehat U(z) z^{d-1} dz \le M .
\end{equation*}

	The constant in \eqref{eq:equicontinuity} is uniform for $m\in[m_-,m_+]$. Integrating from $z$ to $1$ and using the mass bound at $1$ gives
	\[
		V_j(z)^{m_j}\le C\widetilde H_d(z),\qquad
		\widetilde H_d(z)=1+H_d(z)=1+\int_z^1s^{1-d}\,ds,\qquad0<z\le1.
	\]
	Thus $V_j\le C'\widetilde H_d^{1/m_-}$, an integrable majorant against $z^{d-1}dz$: it is bounded for $d=1$, grows logarithmically to a fixed power for $d=2$, and is $O(z^{-(d-2)/m_-})$ for $d\ge3$, with $(d-2)/m_-<d$. This excludes concentration at zero.

Assume, in contradiction, that the tightness at infinity in $L^1(0,\infty, z^{d-1})$ does not hold. Then, there are some $c > 0$, and sequences $R_j \to \infty$ and a subsequence $m_j \to 1$ (which we do not relabel) such that $V_j = U_{\alpha,m_j,M}$
\begin{equation*}
	\int_{R_j}^\infty V_j(\rho) \rho^{d-1} d \rho \ge c.
\end{equation*}
Then we have that
\begin{equation*}
	z^{d-1}(-V_j^{m_j}(z))' \ge \int_z^\infty - \frac{z^{d-1}}{\rho^{d-1}}Q'_{\alpha,m_j}(\tfrac{z}{\rho}) \rho^{d-1} V_j(\rho) d \rho \ge \frac{c}{\Gamma(1-\alpha)} \text{for } z < R_j.
\end{equation*}
Integrating once, we deduce that
\[
	V_j^{m_j}(z) \ge \frac{c}{\Gamma(1-\alpha)} \int_{z}^{R_j} \rho^{1-d} d\rho \quad \text{for all } z < R_j.
\]
\paragraph{Case $d=1$.} Then we have $V_j^{m_j}(z) \ge \frac{c}{\Gamma(1-\alpha)} (R_j - z) \to \infty$ for each $z$ fixed as $j \to \infty$. This contradicts the point-wise convergence.

\paragraph{Case $d=2$.} Then we have $V_j^{m_j}(z) \ge \frac{c}{\Gamma(1-\alpha)} \log \frac{R_j}{z} \to \infty$ for each $z$ fixed as $j \to \infty$. This contradicts the point-wise convergence.

\paragraph{Case $d \ge 3$.} Then we have that $V_j^{m_j}(z) \ge \frac{c}{\Gamma(1-\alpha)(d-2)}(z^{2-d} - R_j^{2-d})$.
This implies that the point-wise limit satisfies
\[
	\widehat U (z) \ge \frac{c}{\Gamma(1-\alpha)(d-2)} z^{2-d}.
\]
Using Fatou's lemma, we deduce that
\[
	\widehat M \ge |\partial B_1| \int_0^\infty \rho^{d-1} \widehat U(\rho) d\rho \ge \frac{c |\partial B_1|}{\Gamma(1-\alpha)(d-2)}\int_0^\infty \rho d \rho = +\infty.
\]
This is a contradiction. Hence, the family $z^{d-1} U_{\alpha,m_j,M}(z)$ is tight in $L^1(0,\infty)$.

Due to the tightness, the uniform bound at zero, and the convergence over compactness, we deduce that as $j \to \infty$ we have that
\[
	\int_0^\infty |U_{\alpha,m_j,M}(z) - \widehat U(z)| z^{d-1} dz \to 0.
\]
Due to the bounds on $Q_{\alpha,m}^{\prime}$ and the convergence $Q_{\alpha,m}^{\prime}\to Q_{\alpha,1}^{\prime}$ as $m\to1$ this is sufficient to prove that
\begin{equation*}
	\widehat U(z)=\int_z^\infty K_{\alpha,1}(z,\rho)\widehat U(\rho)\,d\rho,
	\quad |\partial B_1|\int_0^\infty\widehat U(\rho)\rho^{d-1}\,d\rho=M.
\end{equation*}
Indeed, first pass to \eqref{intdiff. equation} in its integrated form on compact intervals: the kernel relative to $\rho^{d-1}d\rho$ is uniformly bounded away from $\rho=z$ and bounded by $C_z(\rho-z)^{-\alpha}$ near it. Local uniform convergence and tail tightness justify the limit. Integration to infinity, Tonelli's theorem and $\widehat U(\infty)=0$ then give the integral equation; the preceding $L^1$ convergence gives the mass.

By the uniqueness result \Cref{thm:m=1 uniqueness} (applicable by the regularity argument in \Cref{sec:linear uniqueness}) this implies that $\widehat U = U_{\alpha,1,M}.$
For every $m_j \to 1$ the sequence $U_{\alpha,m_j,M}$ is pre-compact in $C_{loc}((0,+\infty))$, and the limit of any convergent subsequence is $U_{\alpha,1,M}$, then the whole sequence converges.
\qed

\section{Extensions and open problems}

\paragraph{Uniqueness of weak solutions with Dirac initial data.}
We expect that the arguments in \cite{Pierre1982} for the uniqueness of the Porous-Medium Equation with measure initial data can be adapted to this setting. The idea of the proof is to study $v = E * u$ (where $E$ is either the Green function for $-\Delta$ if $d \ge 3$ or $-\Delta + \varepsilon I$ if $d =1,2$).

\paragraph{Asymptotic behavior for general initial data.}
We have constructed a family of self-similar profiles $U_{\alpha,m}$.
An open problem remains regarding the stability of these profiles: are they long-time global attractors for the corresponding general Cauchy problem? In the classical case $\alpha=1$, this problem is very well understood \cite{vazquez2014barenblatt}.
For $\alpha \in (0,1)$ and $m = 1$ this was shown in \cite{cortazarHeatEquationMemory2021}.
However, in the time-fractional model, there are no corresponding results in the literature.

\paragraph{The very-singular fast-diffusion problem.}
The range $0<m<m_c$, $d\ge 3$, is usually called very singular fast diffusion. The tools used in this paper are not suitable for treating this range for the equation \eqref{eq:main} with the fractional time derivative. To our knowledge, the problem is completely open.

In the standard case $\alpha=1$ much is known. According to the theory developed, fundamental solutions with finite mass exist but do not have the type of self-similarity used in this paper, there is no mass conservation, and they extinguish in finite time: see \cite{chasseigne2002theory,Vazquez2006,Vazquez2006a, vazquez2017, bonforte2010sharp}.

\paragraph{Fully non-local problems.}
The Porous Medium in non-local pressure \cite{allen2016}.
Another interesting family of problems is $\partial_t^\alpha u = -(-\Delta)^s u^m$. Let the linear case $m = 1$ already be well understood; see \cite{cortazarLargeTimeBehaviorFully2021}.

\paragraph{Numerical analysis for $m_c < m < 1$.}
In this paper, we include a discussion of the scheme for $m > 1$. This case is manageable, as we have compactly supported profiles. From a theoretical point of view, the scheme is also valid for $m_c < m < 1$. In this case, the solutions have power-type tails, and hence smart truncations are needed in order to obtain an implementable scheme.

\paragraph{The asymptotic fan for $0 < m < m_c$ above $U^*$.}
For $\alpha = 1$ it is know that there exists self-similar profiles of the form $U^*(z) ( 1 + z^{-\gamma} + o(z^{-\gamma}))$ that blow up at some $z = z_1$. By re-scaling this creates an ordered family of solutions above $U^*(z)$, each of which blow up at a different $z > 0$.

\paragraph{Uniqueness solutions to \eqref{eq:profile equation} when $m < 1$.}
In this manuscript we have only proven the uniqueness of solutions to \eqref{eq:profile equation} in the case $m > 1$, by taking advantage of the super-linear homogeneity and the compact support, and for $m=1$ by duality. The case $m_c < m < 1$ remains open.

\paragraph{Travelling waves.}
This is a topic where much is known in the standard case and probably nothing in the fractional time derivative case.

\paragraph{Sign-changing self-similar profiles.} In principle, we cannot rule out the possibility that there exist sign-changing solutions to \eqref{eq:profile equation}. We do not find them particularly physically relevant and we will not discuss them further.

\section*{Acknowledgments}

The research of DGC is supported by grants RYC2022-037317-I and  PID2023-151120NA-I00 from Agencia Estatal de Investigación of the Spanish Government. \L{}P has been supported by the National Science Centre, Poland (NCN) under the grant Sonata Bis with a number NCN 2020/38/E/ST1/00153.
JLV was supported by CNS2024-154515 from AEI, the Agencia Estatal de Investigación of the Spanish Government.

		\paragraph{Declaration of AI use.} The first version of this manuscript was written by the authors without AI assistance. The authors subsequently used ChatGPT-6 Astra to assist in revising some of the mathematical arguments in this manuscript. The authors have reviewed and revised the manuscript and take full responsibility for its content.

\paragraph{Conflict Of Interest Statement.} The authors declare no conflict of interest.
\paragraph{Data Availability Statement.} The authors declare no data was used in this manuscript.

\printbibliography

\appendix

\section*{Addresses}

\noindent David Gómez-Castro, \\
Departamento de Matem\'{a}ticas, \\
Universidad Aut\'{o}noma de Madrid,\\ Campus de Cantoblanco, 28049 Madrid, Spain.  \\
e-mail address:~\texttt{david.gomezcastro@uam.es}\\
webpage:  \url{https://gomezcastro.xyz/}

\bigskip

\noindent \L{}ukasz P\l{}ociniczak, \\
Department of Applied Mathematics, \\
Faculty of Pure and Applied Mathematics,\\
Wrocław University of Science and Technology\\
Hugo Steinhaus Center,
ul. Wybrzeże Wyspiańskiego 27,
50-370 Wrocław,
Poland.  \\
e-mail address:~\texttt{lukasz.plociniczak@pwr.edu.pl}\\
webpage:  \url{https://alfa.im.pwr.edu.pl/~plociniczak/}

\bigskip

\noindent Juan Luis V\'azquez, \\
Departamento de Matem\'{a}ticas, \\
Universidad Aut\'{o}noma de Madrid,\\ Campus de Cantoblanco, 28049 Madrid, Spain.  \\
e-mail address:~\texttt{juanluis.vazquez@uam.es}\\
webpage:  \url{https://verso.mat.uam.es/\~juanluis.vazquez/}

\end{document}